\documentclass[10.5pt, a4paper,  reqno, english]{amsart}
\usepackage[all]{xy}
\usepackage[utf8]{inputenc}
\usepackage{amssymb,amsmath,amsthm}
\numberwithin{equation}{section}
\usepackage{amsfonts}
\usepackage{bbm}
\usepackage{mathrsfs}
\usepackage{lmodern}
\usepackage{enumitem}
\usepackage{xcolor}

\usepackage[margin=1.5in]{geometry}

\newcommand{\mr}{\mathbf{r}}
\newcommand{\bu}{\mathbf{u}}

\newtheorem{cor}{Corollary}[section]
\newtheorem{corr}{Corollary}

\newtheorem{lem}{Lemma}[section]

\newtheorem{prop}{Proposition}[section]
\newtheorem{propp}{Proposition}
\newtheorem*{rmk}{Remark}
\newtheorem{thm}{Theorem}[section]
\newtheorem*{thm*}{Theorem}
\newtheorem{thmm}{Theorem}

\newtheorem{conj}{Conjecture}

\begin{document}

\date{\today}
\title{Long large character sums}
\author{Crystel Bujold}
\address{Département de mathématiques et de statistique, Université de Montràal, CP 6128
succ. Centre-Ville, Montréal, QC H3C 3J7, Canada}
\email{bujoldc@dms.umontreal.ca}

\begin{abstract} 
In this paper, we prove a lower bound for $\underset{\chi \neq \chi_0}{\max}\bigg|\sum_{n\leq x} \chi(n)\bigg|$, when $x= \frac{q}{(\log q)^B}$. This improves on a result of Granville and Soundararajan for large character sums when the range of summation is wide. When $B$ goes to zero, our lower bound recovers the expected maximal value of character sums for most characters.

\end{abstract}
\maketitle

 \section{Introduction}
Since their introduction in 1837, Dirichlet characters have played an important role in understanding questions about primes and integers.  As is often the case with multiplicative functions we would like to understand their mean value, in this case, the growth of character sums of the form
\begin{equation}\label{sum chi}
    \sum_{n\leq x} \chi(n),
\end{equation}
where $x$ is a positive real number and $\chi$ is a character modulo an integer $q$. 
The best unconditional  upper bound for (\ref{sum chi}) is given by the P\'olya-Vinogradov inequality (1918):
   \begin{equation*}
    \sum_{n\leq x} \chi(n) \ll \sqrt{q}\log q.
 \end{equation*}
 Montgomery and Vaughan improved this, under the assumption of the  generalized Riemann hypothesis, to
  \begin{equation}\label{M-V}
     \sum_{n\leq x}\chi(n) \ll \sqrt{q}\log\log q.
 \end{equation}
This is best possible (up to the value of the constant), and indeed  Granville and Soundararajan \cite{LargeGS2} proved that 
for any large prime $q$ and   $\theta\in(-\pi,\pi]$ there are  $>q^{1-\frac{c}{(\log\log q)^2}}$ odd characters $\pmod q$ for which 
\begin{equation*}
    \sum_{n\leq x} \chi (n) = e^{i\theta}\frac{e^{\gamma}}{\pi}\sqrt{q}\log\log q +O\left((\log\log q)^{\frac{1}{2}}\right)
\end{equation*}
for almost all  $x\leq q$, where  $\gamma$ is the Euler-Mascheroni constant. In particular this implies that  there are a lot of characters for which
\begin{equation} \label{eq: gamma}
    \underset{x\leq q}{\max}\bigg|\sum_{n\leq x} \chi(n) \bigg| \geq \left(\frac{e^{\gamma}}{\pi}+o(1)\right)\sqrt{q}\log\log q.
    \end{equation}

It is also of interest to understand the behaviour of 
\begin{equation}\label{max chi}
    \underset{\chi \neq \chi_0}{\max}\bigg|\sum_{n\leq x} \chi(n)\bigg|,
\end{equation}
for different values of  $x$. Granville and Soundararajan \cite{largeGS} established lower bounds of (\ref{max chi}) for $x$ covering all ranges up to $q$. For example for small  $x$, they proved that  for any fixed $B>0$
\begin{equation*}
   \max_{\chi \neq \chi_0 }   \bigg|\sum_{n\leq (\log q)^B} \chi(n)\bigg|\gg x \rho(B),
\end{equation*}
where $\rho(B)$ is the Dickman-De Bruijn function.  This is defined by $\rho(u)=1$ for $0\leq u\leq 1$, and $\rho(u)=\frac 1u\int_{u-1}^u \rho(t) dt$ for all $u>1$.

On the other hand, for large $x$ Granville and Soundararajan proved that
\begin{equation}\label{bound 1}
    \underset{\chi \neq \chi_0}{\max}\bigg|\sum_{n\leq \frac{q}{(\log q)^B}}\chi(n) \bigg|\gg \frac{\sqrt{q}}{(\log q)^{\frac{B}{2}-o(1)}}.
\end{equation}
Here we improve this result to the following:
 
\begin{thmm}\label{big thm}
    Let Q be a large integer, for all but at most $Q^{\frac{1}{10}}$ primes $q\leq Q$, if  $0\leq  B <\frac{\log \log \log q}{\log \log \log \log q}$, then
 \begin{equation*}
     \max_{\chi \neq \chi_0 }\bigg|\sum_{n\leq \frac{q}{(\log q)^B}}\chi (n)\bigg|\geq  \frac{1}{\pi}\int_B^\infty \rho(u) du\cdot \sqrt{q} \log \log q +  O(\sqrt{q}\log \log \log q) . 
 \end{equation*}
 \end{thmm}
 
Since $\int_0^\infty \rho(u) du =e^\gamma$ (see \cite{frequency} Lemma 3.3), letting $B$ go to zero recovers \eqref{eq: gamma}. We also believe that Theorem \ref{big thm} should hold for all prime moduli $q$, but we were unable to prove this due to a  the limitation in our Fourier analysis argument. The large character sums in Theorem \ref{big thm} all  arise from odd characters. For even characters we can obtain the following weaker bound,  for all prime moduli $q$.

\begin{thmm}\label{thm even}
    Let q be a large prime and let $1\leq  B <\frac{\log \log \log q}{\log \log \log \log q}$, then\begin{equation*}
    \max_{\substack{\chi \neq \chi_0\\\chi \text{ even}}} \bigg|\sum_{n\leq \frac{q}{(\log q)^B}}\chi(n)\bigg|\geq \frac{\rho(B)}{2}\sqrt{q} + O\left(\frac{\rho(B)\log(B+1) \sqrt{q}(\log\log\log q)^2}{\log \log q}\right).
\end{equation*}
\end{thmm}
The case $B<1$ was excluded from Theorem \ref{thm even} as the focus of this work is Theorem \ref{big thm}, but this could be worked out with extra technicalities.

The Dickman-De Bruijn $\rho$-function appears in both Theorem \ref{thm even} and (\ref{bound 1})
since it counts smooth numbers, and in  both of these ranges ranges, most of the contribution to the character sums comes from the smooth summands.

\begin{conj}\label{conjj}
Let $q$ be large and let $\chi$ be a non-principal character modulo q. If $\log q \leq y\leq x$ 
\begin{equation*}
      \underset{\alpha \in [0,1]}{\max} \bigg|\sum_{\substack{n\leq x \\ P(n)>y}}\frac{\chi(n)}{n}e(\alpha n)\bigg| \ll 1.
\end{equation*}
\end{conj} 
Assuming this conjecture, the  inequalities in Theorem \ref{big thm} and Theorem \ref{thm even} become equalities and therefore, we expect our results to be ``best possible''.

Although our proof of Theorem \ref{big thm} restricts the range to $B\leq \frac{\log\log\log q}{\log\log\log\log q}$, we believe that the lower bound in Theorem \ref{big thm} should extend to a much wider range:

\begin{conj}
 Let q be a large prime and write  $x=\frac{q}{(\log q)^B}$. If $\frac{q}{\exp(\sqrt{\log q} \log\log q)}\leq x$ then
 \begin{equation*}
     \max_{\chi \neq \chi_0 }\bigg|\sum_{n\leq x}\chi (n)\bigg|\gtrsim \frac{1}{\pi}\int_B^\infty \rho(u) du\cdot \sqrt{q} \log \log q.
 \end{equation*}

\end{conj}

A key aspect of our  investigation concerns the following question about lattices.
Given a lattice vector $\mathbf{u} = \frac{1}{M}(u_1,\cdots, u_k) \in (\mathbb{R}/\mathbb{Z})^k$ is there an integer $1\leq \ell\leq M-1$ such that all components of $l\mathbf{u} \pmod 1$ are small?  
The  pigeonhole principle allows us to find many such vectors $\ell\mathbf{u}$, but all of the $\ell$ that we find this way might be even, which for us would mean not being able to work with odd characters. What we need is to find many such odd $\ell$, and so we start by making the following more general definitions:

\begin{equation}\label{cn+}
    C_{n^+}(\eta, k) = \left\{0\leq \ell\leq M-1\text{ , } \ell\equiv 0 \pmod n : |(\ell\mathbf{u})_j|\leq \eta \text{ for  }  1\leq j\leq k \right\},
\end{equation}
and 
\begin{equation}\label{cn-}
    C_{n^-}(\eta, k) = \left\{0\leq \ell\leq  M-1 \text{ , } \ell \not\equiv 0 \pmod n :  |(\ell\mathbf{u})_j|\leq \eta \text{ for  }  1\leq j\leq k \right\},
\end{equation}
where $|x|$ is the distance from $x$ to the nearest integer. These two sets contain all the vector multiples $\ell\mathbf{u}$ with all small components. The pigeonhole principle gives the following:
 
 \begin{propp}\label{even lat}
 Fix a positive real number $N$. Let $\mathbf{u}\in (\mathbb{R}/\mathbb{Z})^k$ with $k$ large for which $M\mathbf{u}\in  \mathbb{Z}^k$. Then for any fixed integer $n<\frac{M}{N^k}$, 
\begin{equation*}
 \#C_{n^+}\left(\frac{1}{N},k\right)  \geq \frac{M}{nN^k}.
\end{equation*}
\end{propp}

Restricting our search to multiples $\ell\mathbf{u}$ with $\ell\not\equiv 0 \mod n$ we obtain the following:

\begin{thmm}\label{small lu}
Fix a positive real number $N$. Let $\mathbf{u}\in (\mathbb{R}/\mathbb{Z})^k$ with $k$ large for which $M\mathbf{u}\in  \mathbb{Z}^k$, and let $n$ be a divisor of $M$. Then either
\begin{enumerate}[label=(\roman*)]
    \item There exists a non-zero vector $\mr \in (\mathbb{R}/\mathbb{Z})^k$ such that $|r_j|\leq k^4N\log^2(N)$ for $j\leq k$ and $n(\mr \cdot \mathbf{u})  \equiv 0 \pmod 1$; or
    \item  
\begin{equation*}
    \#C_{n^-}\left(\frac{2}{N}, k\right)\geq \frac{M}{nN^k}.
\end{equation*}
\end{enumerate}
\end{thmm}

Note that in Proposition \ref{even lat} we can choose $n$ to be any integer, while $n$ is a divisor of $M$ in Theorem \ref{small lu}.

Although we cannot quite show the converse, in the opposite direction, we have:

\begin{thmm}\label{not 2 small}
Let $\mathbf{u}\in (\mathbb{R}/\mathbb{Z})^k$ with $k$ large for which $M\mathbf{u}\in  \mathbb{Z}^k$, and let $n$ be a divisor of $M$. Suppose that there exists $\mr \in \mathbb{Z}^k$ such that $\mr\cdot\mathbf{u} \equiv \frac{t}{n} \pmod 1$ , where $(t,n)=1$. Then for any integer $1\leq \ell\leq M-1$ such that $\ell\not\equiv 0 \pmod n$, the vector  $\mathbf{x} = \ell\mathbf{u} \in \left(\mathbb{R}/\mathbb{Z}\right)^k$ satisfies
\begin{equation}
    |\mr\cdot \mathbf{x} \pmod 1|\geq \frac{1}{n} .
\end{equation}
In particular, if $|\mr|\leq L$, then 
\begin{equation*}
    |\mathbf{x}| \geq \frac{1}{nL}.
\end{equation*}
\end{thmm}

\subsection{Acknowledgement}
This work was done as part of the author's Ph.D. thesis. Major thanks go to Andrew Granville for his supervision and support. We also wish to thank Kevin Ford, James Maynard and Zeev Rudnick for insightful comments on lattices that helped solved a key component of this work.

\section{Smooth numbers}

In this section we collect several important results about smooth numbers.

Hildebrand (\cite{tenenbaum} p.369) prove the following key estimate. For   $\exp((\log\log x)^{5/3+\epsilon}) \leq y=x^{1/u}$ we have 
    \begin{equation*}
        \psi(x,y) = x\rho(u)\left(1 + O\left(\frac{\log(u+1)}{\log y}\right)\right).
    \end{equation*}
 The following estimate for the size of $\rho(u)$ follows from Lemma 3.1 in \cite{frequency}:
\begin{equation}\label{u^u}
    \rho(u) \ll u^{-u} \text{ for all } u\geq 1.
\end{equation}

Corollary 8.3  in \cite{tenenbaum} states that for any integer $k\geq 0$ and real number $u_1>1$, if $u\geq u_1$, then we have 
\begin{equation}\label{rhoprime}
    \rho^{(k)}(u) = (-1)^k \xi(u)^k\rho(u)\left(1+O\left(\frac{1}{u}\right)\right),
\end{equation}
where $\xi(u)$ is the unique real non-zero root of the equation $e^{\xi} = 1+u\xi$.
 Lemma 8.1 in \cite{tenenbaum} states that if  $u \geq 3$ then
\begin{equation}\label{xi}
    \xi(u) = \log(u\log u) + O\left(\frac{\log\log u}{\log u}\right).
\end{equation}
 We immediately deduce that if  $\exp((\log\log x)^{5/3+\epsilon}) \leq y$ a then 
    \begin{equation}\label{saias}
        \psi(x,y) = x\rho(u) + O\left(\frac{x|\rho^{\prime}(u)|}{\log y}\right).
    \end{equation}
We also deduce that peturbing $u$ by a small amount does not affect too much the value of the $\rho$-function.

\begin{lem}\label{rho}
   For $|v|\leq \frac{1}{\log u}$, we have
  \begin{equation*}
       \rho(u+v)= \rho (u)\left(1+ O(|v|\log (u+1)\right).
  \end{equation*}
\end{lem}

\begin{proof}
By the mean value theorem, there exists $u_0 \in [u, u+v)$ such that
\[
\rho(u+v) = \rho(u) + v\rho^\prime(u_0).
\]
By \eqref{rhoprime} and \eqref{xi}, we have 
\[
    \rho^\prime(u_0) \ll  \xi(u_0)\rho(u_0) \ll  \log (u+v)\rho(u).
\]
and the result follows.
\end{proof}
 
With the same proof one can prove  that for $u\geq 2$ and $0 < v \leq \frac{u}{\log u}$, we have
$\rho (u +v) = \rho(u)u^{-(1+o(1))v}$.

The next lemma approximates the sum of reciprocals of $y$-smooth integers using the Dickman-De Bruijn's function. It follows directly from the strong version of Lemma 3.3 in \cite{frequency}. (See remark 3.1)

\begin{lem}\label{smoothlog}
Let $y\geq 2$ and $0 < s\leq r$, then
\begin{equation*}
  \sum_{\substack{y^{s} \leq n\leq y^r \\ P(n)\leq y}} \frac{1}{n} = \log y \int_{s} ^r\rho(t) dt + O(\rho(s)).
\end{equation*}
\end{lem}
\begin{proof}
 First, suppose that $s>1$. Then using partial summation and   \eqref{saias}, we have that
 \begin{align*}
     \sum_{\substack{y^s < n\leq y^r\\P(n)\leq y}}\frac{1}{n} &= \frac{1}{y^r}\psi(y^r, y) - \frac{1}{y^s}\psi(y^s, y) + \int_{y^s}^{y^r}\frac{\psi(t, y)}{t^2}dt\\
     &= O(\rho(s)) + \int_{y^s}^{y^r} \frac{t\rho(u) + O\left(\frac{t|\rho^{\prime}(u)|}{\log y}\right)}{t^2} dt\\
     & = \int_{y^s}^{y^r} \frac{\rho(u) + O\left(\frac{|\rho^{\prime}(u)|}{\log y}\right)}{t} dt
     = \log y\int_s^r \rho(u) du + O(\rho(s)),
 \end{align*}
 changing variable $t = y^u$.
   If $s\leq 1$, then for $s\leq u\leq 1$ we have that $\rho(u) = 1$ and therefore, 
 \begin{align*}
     \sum_{\substack{y^s < n\leq y^r\\P(n)\leq y}}\frac{1}{n}  &= \sum_{y^s < n\leq y} \frac{1}{n} + \sum_{\substack{y < n\leq y^r\\P(n)\leq y}} \frac{1}{n}\\
     &= (1-s)\log y + O(1) + \log y \int_1^r \rho(u) du + O(1)
      \end{align*}
 and the result follows.
\end{proof}

\section{A first result about lattices}\label{lattices}

One of the main challenges in the proof of Theorem \ref{big thm} arises from finding an odd character which takes values close to one on all primes up to some point $T$. In order to handle this obstacle, we prove a corresponding result about lattices.

We say that $\mathbf{u} = \frac{1}{M}(u_1, u_2,\cdots, u_k)\in (\mathbb{R}/\mathbb{Z})^k$ has \emph{order} $M$, if  $M$ is the smallest positive integer for which $M\mathbf{u}\in \mathbb{Z}^k$.

\subsection{The easier case: Vector multipliers $\ell\equiv 0 \pmod n$}
 
As an immediate corollary to Proposition \ref{even lat} we have:

\begin{cor}\label{even cor}
 Let $\mathbf{u}\in (\mathbb{R}/\mathbb{Z})^k$ be a lattice vector of order $M$, then
 \begin{equation*}
    \#C_{2^+}\left(\frac{1}{N},k\right)  \geq  \frac{M}{2N^k}.
\end{equation*}
\end{cor}

 \begin{proof}[\textbf{Proof of Proposition \ref{even lat}}]
Let $\mathbf{x}_\ell \equiv \ell\mathbf{u} \pmod 1$, where $n<\frac{M}{N^k}$ is fixed and the multipliers $0\leq \ell\leq M-1$ satisfies $\ell\equiv 0 \mod n$. We split $\left(\mathbf{R}/\mathbf{Z}\right)^k$ into $N^k$ equal hypercubes, each side of which has length $1/N$. Notice that for each integer $0\leq \ell\leq M-1$, with $l\equiv 0 \pmod n$, the vector $\mathbf{x}_l$ must belong to one of the cubes, and therefore, by the pigeonhole principle, we must have an hypercube $C$ which contains at least
$\frac{M}{nN^k}$ vectors.

Now, fix $\mathbf{x}_{r}\in C$ where $r > s$ for all other vectors $\mathbf{x}_s\in C$. By the construction of the cubes, for any other vector $\mathbf{x_s}$ in $C$  we must have $|x_{r,j} - x_{s,j}|\leq \frac{1}{N}$, for $j\leq k$. Let $\ell=r-s$ and observe that $r-s\equiv 0 \mod n$ and thus the vector $\mathbf{x_{\ell}}= \mathbf{x_{r}} - \mathbf{x_s}\equiv (r-s)\mathbf{u} \pmod 1$ has multiplier $\ell\equiv 0 \pmod n$, with each component of size at most $1/N$. As there are $ \frac{M}{nN^k} $ such vectors $\mathbf{x_s}\in C$, including $\mathbf{x}_r$, we deduce that there are at least $\frac{M}{nN^k}$ integers $0\leq \ell\leq M-1$, with $\ell\equiv 0 \mod n$, such that $\mathbf{x}_\ell$  has components $|x_{\ell,j}| \leq \frac{1}{N}$ for all $j\leq k$ and the result follows. 
\end{proof}
 
 Corollary \ref{even cor} is important for us, as it will allow us to show the existence of many even characters with small argument. However, we need to show that there are a lot of odd characters with small arguments. The next lemma shows that if we can find just one vector with multiplier $\ell \not\equiv 0 \pmod n$ that is small, then  we can find many of them.
 
 \begin{lem}\label{odd set}
 Given a lattice vector $\mathbf{u} \in (\mathbb{R}/\mathbb{Z})^k$ of order $M$, let $\mathbf{x}_\ell \equiv \ell\mathbf{u} \pmod 1$. 
 \end{lem}

  Suppose that $C_{n^-}(\nu, k)\neq\emptyset$, then
  \begin{equation*}
      \#C_{n^-}(\nu+ \eta, k)\geq \#C_{n^+}(\eta, k),
  \end{equation*}
  where the sets are defined as in (\ref{cn+}) and (\ref{cn-}).
  \begin{proof}
 Suppose that there exists an integer $0\leq r\leq M-1$, with $r\not\equiv 0 \mod n$, such that each component of $\mathbf{x}_r$ satisfies $|x_{r,j}|\leq \nu$ for $1\leq j\leq k$. For any integer $s\equiv 0 \pmod n$ in the same range and such that the vector $\mathbf{x}_s \in C_+(\eta,k)$, then $\ell \equiv r-s \pmod M$ satisfies $\ell \not\equiv 0 \pmod n$ and the size of the components of the vector $\mathbf{x}_\ell$ is bounded by 
\[
     |x_{\ell,j}| = |x_{r,j}\pm x_{s,j}|  \leq |x_{r,j}|+ |x_{s,j}|\  \leq \eta + \nu. 
 \]
 Hence $\mathbf{x}_\ell\in C_-(\eta+\nu,k)$.
 Therefore distinct vectors in $C_+(\eta, k)$ will give rise to distinct vectors in $C_-(\eta+\nu, k)$, and therefore it follows that $\#C_-(\eta+\nu, k)\geq \#C_+(\eta, k)$.
 \end{proof}

\subsection{The harder case: Vector multipliers $\ell\not\equiv 0 \pmod n$}
Finding multipliers of the form $\ell \not\equiv 0 \pmod n$ for our lattice vector $\mathbf{u}$ is more subtle; indeed such vectors do not always occur as we see in  Theorem \ref{small lu}.
Theorem \ref{small lu} follows directly from Proposition \ref{even lat}, Lemma \ref{odd set} and the following key proposition.

\begin{prop}\label{prop odd}
Let $N>0$, $k$ be a large integer and let $\mathbf{u}\in (\mathbb{R}/\mathbb{Z})^k$ be a lattice vector of order M. Given a divisor $n$ of $M$ then either
\begin{enumerate}[label=(\roman*)]
    \item There exists a non-zero vector $\mr \in (\mathbb{R}/\mathbb{Z})^k$ such that $|r_j|\leq k^4N\log^2(N)$ for  $j\leq k$ and $n(\mr \cdot \mathbf{u})  \equiv 0 \pmod 1$; or
    \item \begin{equation*}
    C_{n^-}\left(\frac{1}{N}, k\right)\neq \emptyset.
\end{equation*}
\end{enumerate}
\end{prop}

\begin{proof} We will use Fourier analysis to construct a counting function detecting vectors with small components, and apply it to vectors of the form $\mathbf{x}_\ell \equiv \ell\mathbf{u} \pmod 1$.

 For now, suppose that there is a positive real number $L$ for which there is no vector $\mr \in \mathbb{Z}^k$, with $|r_j|<L$ for $j\leq k$  such that $n(\mr \cdot \mathbf{u}) \equiv 0 \pmod 1$. 

 Let  
 \begin{equation*}
    \phi(x) = \begin{cases}
               c_0 e^{\frac{-1}{1-(2x)^2} } & \textrm{ if } x\in \left[-\frac{1}{2}, \frac{1}{2}\right] \\
                0 & \textrm{otherwise,}
                \end{cases}
\end{equation*}
which is a positive valued Schwartz function, 
where $c_0$ is a normalizing constant so that
$ \int_{-\infty}^{\infty} \phi(x) dx = 1$. 

The $k$-dimensional bump function
\begin{equation*}
    \Phi_N(\mathbf{x}) = \prod_{j\leq k}\phi_N(x_j),
\end{equation*}
where $\phi_N(x) := N\phi(Nx)$, is non-negative and has support in $\left[-\frac{1}{2N}, \frac{1}{2N}\right]^k$.
Now let
\begin{equation*}
    F_N(\mathbf{x}) = \sum_{\mathbf{v}\in \mathbb{Z}^k}\Phi_N(\mathbf{v}-\mathbf{x}),
\end{equation*}
and we  define our counting function by
\begin{equation*}
    S(N) = \sum_{\mathbf{x} \in V} F_N (\mathbf{x})
\end{equation*}
for a generic set of vectors $V$. For $V = \{\mathbf{x}\equiv \ell\mathbf{u} \pmod 1: 1\leq \ell\leq M-1 \textrm{ , } \ell \not\equiv 0 \pmod n\}$,  
\begin{equation*}
    S(N)= \sum_{a=1}^{n-1} \sum_{\substack{1\leq \ell\leq M-1\\\ell\equiv a \pmod n}} F_N(\ell\mathbf{u}).
\end{equation*}
If we can show that $S(N)>0$, then $F_N(\ell\mathbf{u})$ is non-zero for some  integer $\ell\not\equiv 0\pmod n$, thus proving the existence of a vector $\mathbf{x}\equiv \ell\mathbf{u} \pmod 1$ with components in $\left(\frac{-1}{N}, \frac{1}{N}\right)$.

 By the Poisson summation formula, we have that

\begin{align*}\label{poisson}
     F_N(\mathbf{x}) =  \sum_{\mathbf{r}\in \mathbb{Z}^k} \frac{e(\mathbf{x}\cdot \mathbf{r})}{N^k}\hat{\Phi}_N\left(\mr\right),
\end{align*}
and so
\begin{align*}
    S(N)&= \sum_{a=1}^{n-1} \sum_{\substack{1\leq \ell\leq M-1\\\ell\equiv a \pmod n}}\sum_{\mathbf{r} \in \mathbb{Z}^k} \frac{e(\mathbf{r}\cdot \ell\mathbf{u})}{N^k}\hat{\Phi}_N\left(\mr\right)\\
        &=\frac{1}{N^k}\sum_{\mathbf{r} \in \mathbb{Z}^k}\hat{\Phi}_N\left(\mr\right)\sum_{a=1}^{n-1}\sum_{0\leq s \leq m-1} e((sn+a)\mathbf{u}\cdot \mathbf{r}).
\end{align*}
where $mn=M$ with $\ell=sn+a$.
The inner sum is 
\begin{equation*}
    e(a\mathbf{u}\cdot \mr)\sum_{0\leq s \leq m-1} e(sn\mathbf{u}\cdot \mathbf{r}).
\end{equation*}
The components of $\mathbf{u}$ are all of the form $\frac{u_j}{M}$ for some $1\leq u_j\leq M-1$, so that we have a complete exponential sum and thus
\begin{equation*}
    \sum_{0\leq s \leq m-1} e(sn\mathbf{u}\cdot \mathbf{r}) = 
        \begin{cases}
            m & \textrm{if }n( \mathbf{u}\cdot \mr) \equiv 0 \pmod 1, \\
            0 & \textrm{otherwise}.
            \end{cases}
\end{equation*}
Hence, we have that
\begin{equation*}
    S(N) = \frac{m}{N^k} \sum_{\substack{r\in \mathbb{Z}^k\\n( \mr\cdot \mathbf{u})\equiv 0 \pmod 1}} \hat{\Phi}_N\left(\mr\right)\sum_{a=1}^{n-1} e\left(a \mr\cdot \mathbf{u}\right).
\end{equation*}
Observe that 
\begin{equation*}
    \sum_{a=1}^{n-1}  e\left(a \mr\cdot \mathbf{u}\right) = \begin{cases}
    n-1 & \textrm{ if } \mr \cdot u \equiv 0 \pmod 1,\\
    -1 & \textrm{otherwise}.
    \end{cases}
\end{equation*}
Therefore, it follows that 
\begin{equation*}
    S(N) = \frac{m}{N^k}\Bigg((n-1)\sum_{\substack{\mr \in \mathbb{Z}^k \\ \mr \cdot \mathbf{u} \equiv 0 \pmod 1}}\hat{\Phi} _N(\mr) - \sum_{\substack{\mr \in \mathbb{Z}^k \\ n(\mr \cdot \mathbf{u}) \equiv 0 \pmod 1\\\mr\cdot\mathbf{u}\not\equiv 0 \pmod 1}}\hat{\Phi} _N(\mr)\Bigg).
\end{equation*}
Only the small values of $\mr$ make a significant contribution, so that we can truncate the sums without too much loss. Indeed,
\begin{align*}
   \Bigg| \sum_{\substack{\mr \in \mathbb{Z}^k \\\underset{j\leq k}{\max} |r_j|>L\\ \mr \cdot \mathbf{u} \equiv 0 \pmod 1}}\hat{\Phi} _N(\mr) - \sum_{\substack{\mr \in \mathbb{Z}^k \\ \underset{j\leq k}{\max} |r_j|>L \\ n(\mr \cdot \mathbf{u}) \equiv 0 \pmod 1\\\mr\cdot\mathbf{u}\not\equiv 0 \pmod 1}}\hat{\Phi} _N(\mr)\Bigg| 
   &\leq\sum_{j=1}^k \sum_{\substack{\mr \in \mathbb{Z}^k \\ |r_j|>L}} |\hat{\Phi}_N(\mr)|\\
   & \leq k \sum_{\substack{r \in \mathbb{Z} \\|r|>L}} |\hat{\phi}_N(r)|\left( \sum_{r \in \mathbb{Z}} |\hat{\phi}_N(r)|\right)^{k-1}\\
   &\leq k \int_{|t|>L} \bigg|\hat{\phi}\left(\frac{t}{N}\right)\bigg|dt\left(\int_{-\infty}^{\infty} \bigg|\hat{\phi}\left(\frac{t}{N}\right)\bigg|dt \right)^{k-1}\\
   \end{align*}
 Since $\phi(x)$ is a Schwartz function, so is $\hat{\phi}(y)$ which means that it is in $L^1$ and thus, with the appropriate change of variable,
  \begin{equation*}
      \int_{-\infty}^{\infty} \bigg|\hat{\phi}\left(\frac{t}{N}\right)\bigg| dt= cN,
  \end{equation*}
   for some absolute constant c.
  Moreover, knowing that $|\hat{\phi}(y)|\ll \frac{e^{-\sqrt{y}}}{y^{3/4}}$ (see \cite{John}) for large enough positive $y\in \mathbb{R}$, we have
  \[
       \int_{|t|>L} \bigg|\hat{\phi}\left(\frac{t}{N}\right)\bigg|dt 
       \ll \int_{|t|>L} \frac{e^{-\sqrt{\frac{|t|}{N}}}}{\left(\frac{|t|}{N}\right)^{\frac{3}{4}}}dt
       = N\int_{u>L/N} \frac{e^{-\sqrt{u}} }{u^{\frac{3}{4}}}du
       \ll  \bigg( \frac{N^7}{L^3}\bigg)^{1/4} e^{-\sqrt{\frac{L}{N}}}.
  \]
Hence, putting this together we have 
\begin{align*}
     k \int_{|t|>L} \bigg|\hat{\phi}\left(\frac{t}{N}\right)\bigg|dt\left(\int_{-\infty}^{\infty} \bigg|\hat{\phi}\left(\frac{t}{N}\right)\bigg|dt \right)^{k-1} \ll k (Nc)^{k-1} \bigg( \frac{N^7}{L^3}\bigg)^{1/4} e^{-\sqrt{\frac{L}{N}}},
\end{align*}
and choosing $L = k^4N\log^2(N)$, we get that
\begin{align*}
       \bigg| \sum_{\substack{\mr \in \mathbb{Z}^k \\\underset{j\leq k}{\max} |r_j|>L\\ \mr \cdot \mathbf{u} \equiv 0 \pmod 1}}\hat{\Phi} _N(\mr) - \sum_{\substack{\mr \in \mathbb{Z}^k \\ \underset{j\leq k}{\max} |r_j|>L \\ n(\mr \cdot \mathbf{u}) \equiv 0 \pmod 1\\\mr\cdot\mathbf{u}\not\equiv 0 \pmod 1}}\hat{\Phi} _N(\mr)\bigg|  
        &\ll \frac{c^{k}}{N^{k^2-k}}.
\end{align*}
Therefore we can truncate the sum to get  
\begin{equation*}
    S(N) = \frac{m}{N^k}\Bigg((n-1)\sum_{\substack{\mr \in \mathbb{Z}^k \\ |r_j|<L \\ \mr \cdot \mathbf{u} \equiv 0 \pmod 1}}\hat{\Phi} _N(\mr) - \sum_{\substack{\mr \in \mathbb{Z}^k \\|r_j|<L \\ n(\mr \cdot \mathbf{u}) \equiv 0 \pmod 1\\\mr\cdot\mathbf{u}\not\equiv 0 \pmod 1}}\hat{\Phi} _N(\mr)+ o(1)\Bigg).
\end{equation*}

Now, by hypothesis, there are no non-zero vectors $|r_j|\leq  L$ satisfying $n(\mr\cdot \mathbf{u})\equiv 0 \pmod 1$, and so
\begin{equation*}
    S(N) = \frac{m}{n^k}\left((n-1)\hat{\Phi}_N(0) + o(1) \right)=\frac{m}{n^k} (n-1 + o(1))>0
\end{equation*}
as $\hat{\Phi}_N(0) =  (\int_{-\infty}^{\infty}\phi_N(t) dt )^k\ = 1$.
This means that there is an integer $\ell \not\equiv 0 \pmod n$ with $1\leq \ell\leq M-1$, for which $F_N(\ell\mathbf{u}) > 0$; in other words,   if $\mathbf{x}\equiv \ell\mathbf{u} \pmod 1$, then $|x_j|< \frac{1}{N}$ for all $1\leq j\leq k$.
\end{proof}

We highlight the case $n=2$ as it will play a role in the proof of Theorem \ref{big thm}.

\begin{corr}\label{odd cor}
 Let $\mathbf{u}\in (\mathbb{R}/\mathbb{Z})^k$ be a vector of order $2m$, and suppose that there is no vector $\mr\in \mathbb{R}^k$ with $|r_j|<k^4N(\log N)^2$ for all $j\leq k$,  such that $2(\mr\cdot \mathbf{u})\equiv 0 \mod 1$, then 
 \begin{equation*}
    C_{2^-}\left(\frac{2}{N}, k\right)\geq \frac{M}{2N^k}.
\end{equation*}.
\end{corr}

Next we establish the complementary Theorem  \ref{not 2 small}:

 \begin{proof} [Proof of Theorem  \ref{not 2 small}]
 Suppose that $\mr\cdot\mathbf{u}\equiv \frac{t}{n} \pmod 1$. If $\mathbf{x}  = \ell \mathbf{u}$  then
 \begin{align*}
     \mr\cdot \mathbf{x}  = \ell(\mr\cdot\mathbf{u})   \equiv    \frac{\ell t}{n} \pmod 1.
 \end{align*}
 Since $(t,n) = 1$ and $\ell \not\equiv 0 \pmod n$, we deduce that 
 $\ell t\not\equiv 0 \pmod n$, and therefore 
 \begin{equation*}
      |\mr\cdot \mathbf{x} \pmod 1| \geq \frac{1}{n} ,
 \end{equation*}
 which proves the first part of the theorem.

 The second part follows directly from the observation that
 
 \begin{align*}
  |\mr||\mathbf{x}|\geq    |\mr||\mathbf{x}|\cos\theta = \mr\cdot \mathbf{u} \geq \frac{1}{n},
 \end{align*}
 so that 
 \begin{equation*}
     |\mathbf{x}|\geq     \frac{1}{|\mr|n}. \qedhere
 \end{equation*}
\end{proof}

In the next section, we apply these results o the character setting to obtain important information on character sums that will be necessary in the proof of Theorem \ref{big thm}.

\section{When $\chi$ pretends to be 1}\label{chap pretends}
In order to derive a lower bound for our character sum, we would like to find a character that pretends to be 1, that is to say a character taking values close to 1 on all the small primes. It is believed that there are characters taking value 1 for all the primes $p\ll (\log q)^{1-\epsilon}$, but showing this is out of reach, so we resort to a softer condition. Instead, we will consider   the sets

\begin{equation}\label{A_1}
    A_{\pm}(T,N) = \left\{\chi \pmod q: \chi(-1)=\pm1 \textrm{ , }  \underset{p\leq T}{\max}|\chi(p)-1| \ll \frac{1}{N}\right\},
    \end{equation}
where $T\geq 2$ and $N = N(T) \rightarrow \infty$ as $T\rightarrow \infty$.

 In this section, we investigate properties of characters that belong to $A_\pm(T,N)$, and then proceed to confirm that the sets $A_\pm(T,N)$ do indeed contain many characters.

\subsection{What if $\chi$ pretends to be 1?}

\begin{prop}\label{hyp+}
  Suppose that $\chi \in A_\pm(T,N)$ , and let $\log\log y = \left(1+ O\left(\frac{1}{N}\right)\right)\log\log T$, then 
  \begin{equation*}
      \sum_{p\leq y}\frac{\chi(p)-1}{p} \ll h(T) \text{ where } h(T) := \frac{\log\log T}{N}.
  \end{equation*}
  \begin{proof}
  Using the bound from (\ref{A_1}), $\underset{p\leq T}{\max}|\chi(p)-1| \ll \frac{1}{N}$, we have
    \begin{align*}
        \sum_{p\leq y}\frac{\chi(p)-1}{p}& =  \sum_{p\leq T}\frac{\chi(p)-1}{p} + O\left(\log\left(\frac{\log y}{\log T}\right)\right)\\
         &\leq  \underset{p\leq T}{\max} |\chi(p)-1|\sum_{p\leq T}\frac{1}{p}+ O\left(\frac{\log\log T}{N}\right)  \ll\frac{\log\log T}{N}.
    \end{align*}
  \end{proof}
\end{prop}

 Now Proposition \ref{hyp+} allows us to show that we can indeed approximate $\chi$ by 1 when performing logarithmic sums.

\begin{prop}\label{switch to f}
 Suppose that $\chi \in A_\pm(T,N)$  and let $f(n)$ be any bounded function. Let $y>T$ be such that $\log\log y = \left(1+ O\left(\frac{1}{N}\right)\right)\log\log T$ and let $0\leq u\leq u^{\prime}< \exp((\log y)^{3/5 -\epsilon})$. Then writing $w = \max\{0, u-1\}$ and $w^{\prime} = \max\{u^{\prime} -u, u^{\prime} -1\}$ we have
 \begin{equation*}
     \bigg|\sum_{\substack{ y^u\leq n\leq y^{u^{\prime}} \\ P(n)\leq y}}\frac{\chi(n)}{n}f(n) -\sum_{\substack{ y^u\leq n\leq y^{u^{\prime}} \\ P(n)\leq y}}\frac{f(n)}{n} \bigg| 
        \ll  h(T) \log y \int_{w}^{w^{\prime}} \rho(t)dt .
\end{equation*}

\end{prop}
\noindent We will need the following lemmas.

\begin{lem}\label{alpha}
    Let $|\alpha|\leq 1$, then
    
\begin{equation*}
    |\alpha \beta - 1| \leq |\beta-1| + |\alpha -1|.
\end{equation*}  
\end{lem}
\begin{proof}
    Observe that 
    
    \begin{align*}
        |\alpha\beta-1| \leq |\alpha\beta - \alpha| +|\alpha-1|
        \leq |\beta -1| + |\alpha-1|.\qquad  \qedhere
    \end{align*}
\end{proof}
As an immediate corollary, by complete multiplicativity of characters, we obtain

\begin{cor}\label{cor chi}
\begin{equation*}
    |\chi(n)-1| \leq \sum_{p^k|n}|\chi(p)-1|
\end{equation*}
\end{cor}

 \begin{lem}\label{sum p^k}
Let $y>T$ be such that $\log\log y = \left(1+ O\left(\frac{1}{N}\right)\right)\log\log T$ and assume that $\chi \in A_\pm(T,N)$. Then
 \begin{equation*}
     \sum_{\substack{p\leq y\\k \geq 1}} \frac{|\chi(p)-1|}{p^k} \ll h(T)
 \end{equation*}
 
 \begin{proof}
 Since the sum over $k$ is a geometric series, we have
 \begin{align*}
      \sum_{\substack{p\leq y\\k \geq 1}} \frac{|\chi(p)-1|}{p^k}
      =  \sum_{p\leq y }\frac{|\chi(p)-1|}{p-1}
      \leq 2\sum_{p\leq y} \frac{|\chi(p)-1|}{p}
      \ll h(T)
      \end{align*}
      by Proposition \ref{hyp+}.
 \end{proof}
 \end{lem}

\begin{proof}[\textbf{Proof of Proposition \ref{switch to f}}]
   Start with 
    \begin{align*}
        \bigg|\sum_{\substack{y^u\leq n\leq y^{u^{\prime}} \\ P(n)\leq y}}\frac{\chi(n)}{n}f(n) -\sum_{\substack{y^u\leq n\leq y^{u^{\prime}} \\ P(n)\leq y}}\frac{f(n)}{n} \bigg| 
        &\ll \sum_{\substack{y^u\leq n\leq y^{u^{\prime}}\\P(n)\leq y} }\frac{|\chi(n)-1|}{n},
    \end{align*}
then using Corollary \ref{cor chi}, we have
\begin{align*}
   \sum_{\substack{y^u\leq n\leq y^{u^{\prime}}\\P(n)\leq y}}\frac{|\chi(n)-1|}{n} &\leq \sum_{\substack{y^u\leq n\leq y^{u^{\prime}} \\ P(n)\leq y}}\frac{1}{n}\sum_{p^k|n}|\chi(p) -1|
   = \sum_{\substack{p\leq y\\ k\geq 1}}  \frac{|\chi(p)-1|}{p^k}
   \sum_{\substack{\frac{y^u}{p^k}\leq m\leq \frac{y^{u^{\prime}}}{p^k}\\ P(m)\leq y}}\frac{1}{m}\\
   & =\sum_{\substack{p\leq y\\ k\geq 1}}\frac{|\chi(p)-1|}{p^k}\left(\log y\int_{u-v_p}^{u^{\prime}-v_p}\rho(t) dt + O(\rho(u-v_p))\right)
     \end{align*}
   by Lemma \ref{smoothlog}, with $v_p = k\frac{\log p}{\log y}$ . Now if $p^k\leq y$, then $v_p\leq \min\{u,1\}$ and if $p^k > y$, then as $p\leq y$, we must have $k\geq 2$ and therefore $y^{1/k}< p \leq y$. So next we split the sum to cover these two cases.
\begin{align*}
    \sum_{\substack{y^u\leq n\leq y^{u^{\prime}}\\P(n)\leq y}}\frac{|\chi(n)-1|}{n}  &\leq \log y\left[ \sum_{p^k\leq y}\frac{|\chi(p)-1|}{p^k} \left(\int_{\max\{0, u-1\}}^{\max\{u^{\prime}-u, u^{\prime}-1\}}\rho(t) dt +O\left(\frac{\rho(u-1)}{\log y}\right)\right) \right.\\
    &+ \left.\sum_{k\geq 2} \sum_{y^{1/k}< p\leq y} \frac{|\chi(p)-1|}{p^k}\left(\int_{u-v_p}^{u^{\prime}-v_p}\rho(t) dt + O\left(\frac{\rho(u-v_p)}{\log y}\right)\right)\right].
\end{align*}   
  Bounding the second sum trivially, we have 
  
 \begin{align*}
     \sum_{k\geq 2} \sum_{y^{1/k}< p\leq y} \frac{|\chi(p)-1|}{p^k}\left(\int_{u-v_p}^{u^{\prime}-v_p}\rho(t) dt + O\left(\frac{\rho(u-v_p)}{\log y}\right)\right)& \ll \int_2^{\log y} \int_{y^{1/k}}^y \frac{1}{t^k} dt\\
     & \ll \frac{\log y}{\sqrt{y}}.
 \end{align*}
 and using Lemma \ref{sum p^k} to bound the first sum, we get
  
  \begin{equation*}
         \sum_{\substack{y^u\leq n\leq y^{u^{\prime}}\\P(n)\leq y}}\frac{|\chi(n)-1|}{n} \ll h(T)\log y\int_{w}^{w^{\prime}}\rho(t) dt +\frac{\log^2 y}{\sqrt{y}},
  \end{equation*}
  where $w = \max\{0, u-1\}$ and $w^{\prime} = \max\{u^{\prime} -u, u^{\prime} -1\}$.
\end{proof}

The bound on the characters in $A_\pm(T,N)$  also allows us to evaluate logarithmic character sums over $y$-smooth numbers. So next we show

\begin{prop}\label{s1}
  Assume that $\chi \in A_\pm(T,N)$, then for $y\geq T$ with $\log\log y = \left(1+ O\left(\frac{1}{N}\right)\right)\log\log T$  and $B < \exp((\log y)^{3/5 -\epsilon})$, we have   
\begin{equation*}
    \sum_{\substack{n>y^{B}\\P(n)\leq y}} \frac{\chi(n)}{n} = \log y \int_{B}^{\infty} \rho(u) du  +  O(1 + h(T)\log y).
\end{equation*}
\end{prop}

We start by writing the sum as

\begin{equation}\label{split}
   \sum_{\substack{n>y^{B}\\P(n)\leq y}} \frac{\chi(n)}{n}=     \sum_{\substack{n\geq 1\\P(n)\leq y}}\frac{\chi(n)}{n} - \sum_{\substack{n\leq y^{B}\\P(n)\leq y}}\frac{\chi(n)}{n},
\end{equation}
and we first use Proposition \ref{hyp+} to evaluate the first sum on the right hand side of (\ref{split}).

\begin{lem}\label{whole sum}
Assume that $\chi \in A_\pm(T,N)$, then for $y\geq T$ with $\log\log y = \left(1+ O\left(\frac{1}{N}\right)\right)\log\log T$,  
\begin{equation*}
    \sum_{\substack{n\geq 1\\P(n)\leq y}}\frac{\chi(n)}{n} = e^{\gamma}\log y + O\left(h(T)\log y\right).
\end{equation*}

\begin{proof}
Taking the Euler product, we have
\begin{equation*}
    \sum_{\substack{n\geq 1\\P(n)\leq y}}\frac{\chi(n)}{n} = \prod_{p\leq y} \left( 1-\frac{\chi(p)}{p}\right)^{-1}.
\end{equation*}
Now taking absolute values we have
\begin{align*}
    0\leq \log    \Bigg|\frac{\prod_{p\leq y} \left( 1-\frac{\chi(p)}{p}\right)^{-1}}{\prod_{p\leq y} \left( 1-\frac{1}{p}\right)^{-1}}\Bigg|=  \bigg|\sum_{p\leq y}\sum_{k\geq 1}\frac{\chi(p)^k-1}{kp^k}\bigg| 
    \leq  \sum_{p\leq y}\sum_{k\geq 1}\frac{|\chi(p)^k-1|}{kp^k} .
\end{align*}
Applying Lemma \ref{alpha} and computing the geometric series, we get, using the fact that $\underset{p\leq T}{\max}\, |\chi(p)-1|\ll \frac{1}{N}$, that this is 
\begin{align*}
   \leq \sum_{p\leq y}\sum_{k\geq 1}\frac{|\chi(p)-1|}{p^k} \leq \exp\left(\sum_{p\leq y} \frac{|\chi(p)-1|}{p-1}\right)\ll h(T).
\end{align*}
Using Mertens estimate, we deduce the result.  
\end{proof}
\end{lem}
\begin{proof}[\textbf{Proof of Proposition \ref{s1}}]
 Starting with (\ref{split}) and using Lemmas \ref{whole sum} and Proposition \ref{switch to f} with $f \equiv 1$, we get 

\begin{align*}
       \sum_{\substack{n>y^{B}\\P(n)\leq y}} \frac{\chi(n)}{n}&= \sum_{\substack{n\geq 1\\P(n)\leq y}}\frac{\chi(n)}{n} - \sum_{\substack{n\leq y^{B}\\P(n)\leq y}}\frac{\chi(n)}{n}\\
       & = e^{\gamma}\log y +   O\left( h(T)\log y\right) -\sum_{\substack{n\leq y^{B}\\P(n)\leq y}}\frac{1}{n} + O\left(h(T)\log y \int_0^B \rho(u)du\right).
\end{align*}

Now using Lemma \ref{smoothlog}, we have

\begin{align*}
        \sum_{\substack{n>y^{B}\\P(n)\leq y}} \frac{\chi(n)}{n} & = e^{\gamma}\log y - \log y \int_0^{B} \rho(u)du + O(1)+ O(h(T)\log y)\\
       & = \log y \int_{B}^{\infty}\rho(u) du + O(1+ h(T)\log y),
\end{align*}
as $\int_0^{\infty} \rho(u)du= e^{\gamma}$.
\end{proof}

 \subsection{Finding 1-pretentious characters: incursion in the world of lattices}
  It remains to show that we can find characters that belong to $A_\pm(T,N)$. In order to do so, we turn to our theorems on lattices from section \ref{lattices}.  We start with the set containing even characters and we show the following bound which holds for all prime moduli $q$. 

 \begin{prop}\label{big set}
 Let $N\geq 1$ and $T\geq 3$. Then
\begin{equation}
    \bigg|A_+(T,N)\bigg|\geq \frac{\phi(q)}{2N^{\pi(T)}}.
\end{equation}
 In particular, if $q$ is prime, then 
\begin{equation}
    \bigg|A_+(T,N)\bigg|\gg \frac{q}{N^{2T/\log T}}.
\end{equation}

 \begin{proof}  Let  $\theta_p = \theta_p(\chi)= \frac{\arg\big(\chi(p)\big)}{2\pi}$, and observe that
\begin{equation*}
    |\chi(p)-1| = 2\pi|\theta_p| + O(\theta_p^2),
\end{equation*}
 so that the bound in (\ref{A_1}) is equivalent to showing that $\underset{p\leq T}{\max} |\theta_p|\ll \frac{1}{N}$ so we are looking for a lower bound on the size of
\begin{equation*}
    C_+\left(\frac{1}{N}, T\right) = \left\{\chi \pmod q:\chi(-1)= 1,  |\theta_p| \leq \frac{1}{N} \text{  }\forall p\leq T\right\}.
\end{equation*}
 So we let $k = \pi(T)$, we choose a generator $\chi$ for the group of characters and we consider the $k$-dimensional argument vector
\begin{equation*}
    V_{\chi} = (\theta_2, \theta_3, \dots, \theta_{p_k})\in\left(\mathbf{R}/\mathbf{Z}\right)^k.
\end{equation*}

As the subgroup of even characters arises from taking $\chi^\ell$ for even integers $1\leq \ell \leq \phi(q)$, then each even character has an argument vector given by $\ell V_\chi \pmod 1$ for some even $1\leq \ell\leq \phi(q)$. 
 
Now, as $\chi$ has order $\phi(q)$ in the group of characters, then the lattice vector $V_{\chi}$ must have order $d$, where $d| \phi(q)$.  However, since $\chi^\ell$ produces distinct characters for each $1\leq \ell \leq \phi(q)$, then for every integer $1\leq \ell \leq d$, there must be $\frac{\phi(q)}{d}$ characters $\psi = \chi^r$ for which $rV_\chi \equiv \ell V_\chi \pmod 1$, and choosing to view each of these as distinct vectors and we may consider the vector $V_\chi$ to have order $M = \phi(q)$. That is, taking $\mathbf{u} = V_\chi$, by Corollary \ref{even cor} from section \ref{lattices}, we get that
\begin{align*}
    \bigg|C_+\left(\frac{1}{N}, T\right)\bigg| &= \bigg|C_{2^+}\left(\frac{1}{N}, k\right)\bigg| \geq \frac{\phi(q)}{2N^k}.
\end{align*}
It follows that
\begin{equation*}
    \max_{p\leq T}|\chi(p)-1|\ll \frac{1}{N}
\end{equation*}
for at least $ \frac{\phi(q)}{2N^k}$ even characters $\pmod q$, which proves the first part of the proposition. The second part of Proposition \ref{big set} is immediate.
\end{proof}
\end{prop}

For the set containing the odd characters, we obtain a slightly a weaker result which holds for most of the prime moduli $q$ except for a small exceptional set. This limitation comes from our inability to exploit fully the Fourier analysis argument in Theorem \ref{small lu}  and improving this argument by removing or improving the dependence on $k$ in the upper bound for $|r_j|$ would lead to a result holding for all prime moduli $q$.

 \begin{prop}\label{hyp odd}
Let $Q$ be a large integer and let $T\leq \frac{\log Q}{100}$ and $N\leq \frac{T}{2(\log T)^3}$. For all but at most $Q^{\frac{1}{10}}$ primes $q\leq Q$,
 \begin{equation}
    \bigg|A_-(T,N)\bigg|\gg \frac{q}{N^{2T/\log T}}.
 \end{equation}
 \end{prop}

 As in Proposition \ref{big set}, the strategy to prove Lemma \ref{hyp odd} will be to use our theorems on lattices from section \ref{lattices}. In particular, the proposition will follow from Corollary \ref{odd cor} and in order to get the desired bound, we will be required to show that for most primes $q\leq Q$, there are no small vector $\mr \in \mathbb{Z}^k$ such that $2(\mr \cdot V_\chi) \equiv 0 \pmod 1$. This is the purpose of the following Lemma.

\noindent So again, let
\begin{equation*}
    \mathbf{V}_\chi(k) = (\theta_2,\theta_3, \cdots, \theta_{p_k}) \text{ where each } \theta _{p_j} = \frac{\arg(\chi(p_j))}{2\pi} .
\end{equation*}

\begin{lem}\label{size r}
    Let $Q$ be a large integer and $k\leq \frac{1}{60} \frac{\log Q}{\log \log Q}$. Let $\chi \pmod q$ be a character of order $q-1$ and let $\mathbf{u}_q = \mathbf{V}_\chi(k)$.
    For all but at most $Q^{\frac{1}{10}}$ primes $q\leq Q$, if $n(\mr\cdot \mathbf{u}_q)\equiv 0 \mod 1$, for some integer $n\leq Q^{\frac{1}{160}}$, then there exists $j\leq k$ such that
$  |r_j|> k^5$.
\end{lem}
 
  \begin{proof}
   For given prime $q$ and  $\chi \pmod q$ generating the group of character, let $\mathbf{u}_q = \mathbf{V}_\chi(k)$ be the argument vector. Define 
    \begin{equation*}
        S(Q)= \left\{\frac{Q}{2}<q\leq Q: \exists \mr \in \mathbb{Z}^k \textrm{ with } |r_j|\leq k^5 \textrm{ and } n(\mr\cdot \mathbf{u}_q) \equiv 0 \mod 1 \textrm{, } n\leq Q^{\frac{1}{160}}\right\}.
    \end{equation*}
 We will now show that   $\#S(Q)\leq Q^{\frac{1}{12}}$ which implies that for most primes $q$, the condition $n(\mr\cdot \bu_q)\equiv 0 \mod 1$ implies that the components of $\mr$ are greater than $k^5$.
 
 First, suppose that $q\in S(Q)$ and consider
 \begin{align*}
         \chi\left(\prod_{j\leq k}p_j^{r_jn}\right)  = \prod_{j\leq k}\chi(p_j)^{r_jn} = e^{2\pi i n(\bu_q\cdot \mr)} = e^{0}  = 1.
 \end{align*} 
 As $\chi$ is a generator for the group of characters, we deduce that $\prod_{j\leq k}p_j^{r_jn} \equiv 1 \pmod q$, which means that
 \begin{equation*}
     \prod_{r_j>0}p_j^{r_jn}\equiv \prod_{r_i<0}p_i^{|r_i|n} \pmod q,
 \end{equation*}
  from which we deduce that
 \begin{equation}\label{product}
     q \text{ divides }     \prod_{r_j>0}p_j^{r_jn}- \prod_{r_i<0}p_i^{|r_i|n}.
 \end{equation}
 
 \noindent Now fixing $n$ and $r$, we wish to count the number of primes for which (\ref{product}) can hold. So let 
 \begin{equation*}
     s(r,n) = \#\left\{ \frac{Q}{2}< q < Q :  q \text{ divides }     \prod_{r_j>0}p_j^{r_jn}- \prod_{r_i<0}p_i^{|r_i|n}
\right\},
 \end{equation*}
 and observe that
 \begin{equation*}
     \prod_{q \in s(r,n)}q \text{ divides }  \prod_{r_j>0}p_j^{r_jn}- \prod_{r_i<0}p_i^{|r_i|n},
 \end{equation*}
 so that
  \begin{equation*}
     \prod_{q \in s(r,n)}q \leq \Bigg|\prod_{r_j>0}p_j^{r_jn}- \prod_{r_i<0}p_i^{|r_i|n}\Bigg|.
 \end{equation*}

\noindent  Using the lower bound on $q$, we have that
 \begin{align*}
      \left(\frac{Q}{2}\right)^{\#s(r,n)} &\leq  \Bigg|\prod_{r_j>0}p_j^{r_jn}- \prod_{r_i<0}p_i^{|r_i|n}\Bigg| \leq \prod_{j\leq k}p_j^{|r_j|n}\\
     & \leq \left(\prod_{j\leq k}p_j\right)^{k^5Q^{\frac{1}{160}}}  \leq e^{k\log k(1+o(1))k^5Q^{\frac{1}{160}}}  \leq e^{2k^6\log kQ^{\frac{1}{160}}}.
 \end{align*}
  It follows that 
 \begin{equation*}
     \#s(r,n)\leq \frac{2k^6\log kQ^{\frac{1}{160}}}{\log \left(\frac{Q}{2}\right)}.
 \end{equation*}
  Now summing over all values of $n$ and possible $\mr$ we get that 
 \begin{align*}
     \#S(Q) &= \sum_{n,\mr} s(r,n)  \leq \sum_{\substack{|r_j|\leq k^5\\j\leq k}}\sum_{n\leq Q^{\frac{1}{160}}} k^6\log k\, Q^{\frac{1}{160}}\\
     &\leq (2k^5 +1)^kk^6\log k\, Q^{\frac{1}{80}} \leq (3k)^{5k}Q^{\frac{1}{80}} \leq Q^{\frac{1}{12}}Q^{\frac{1}{80}}=Q^{\frac{23}{240}} 
 \end{align*}
 since 
 \begin{align*}
     (3k)^{5k}\leq \left(\frac{3}{60}\frac{\log Q}{\log\log Q}\right)^{\frac{1}{12}\frac{\log Q}{\log \log Q}} \leq \exp\left(\log\log Q \frac{\log Q}{12\log\log Q}\right)  =Q^{\frac{1}{12}}.
 \end{align*}
 as $k\leq \frac{1}{60}\frac{\log Q}{\log\log Q}$.
  
Finally, the number of exceptional   primes $q\leq Q$ is
\begin{align*}
     \sum_{l=0}^{\infty} S\left(\frac{Q}{2^l}\right) 
     \leq \sum_{l=0}^{\infty} \left(\frac{Q}{2^l}\right)^{\frac{23}{240}} = Q^{\frac{23}{240}}\sum_{l=0}^{\infty} \left(\frac{1}{2^{\frac{23}{240}}}\right)^{l} \leq Q^{\frac{1}{10}}. \qquad \qedhere
\end{align*}
 \end{proof}

 With this restriction on the vector $\mr$ at our disposition, we now prove Proposition \ref{hyp odd}.
 
 \begin{proof}[\textbf{Proof of Proposition \ref{hyp odd}}]
  Let $q$ be a prime and let  $\mathbf{u}_q = \mathbf{V}_{\psi}(k) = (\theta_{p_1}, \cdots, \theta_{p_k})$ be the argument vector for $\psi$, where $\psi$ is chosen to be a generator for the group of characters $\pmod q$.  Because $\psi$ has order $\phi(q)=q-1$ in the group of characters, we view $\bu_q$ as a vector of order $q-1 = 2m$. Now, as in the even case, Proposition \ref{hyp odd} is equivalent to finding a lower bound for
\begin{equation*}
    C_-(\nu, T) = \left\{\chi \pmod q:\chi(-1)=- 1,  |\theta_p| \leq \nu, \forall p\leq T\right\},
\end{equation*}
for $\nu \ll \frac{1}{N}$.

 Letting $k = \pi(T)$ be the number of primes up to $T$, we observe that taking $d= 2$ as the divisor of the order $q-1 = 2m$, we have
   \begin{equation*}
        C_-\left(\nu, T\right) = C_{2^-}\left(\nu, k\right).
   \end{equation*}
That is, by Corollary \ref{odd cor}, we have that
  \begin{equation*}
      \bigg| C_-\left(\frac{2}{N}, T\right)\bigg| \geq \frac{2m}{2N^k}
  \end{equation*}
  provided that there are no vector $\mr\in \mathbb{Z}^k$, with $|r_j|\leq k^4N\log^2N$ for all $j\leq k$, such that $2(\mr\cdot \bu_q)\equiv 0 \pmod 1$. But Lemma \ref{size r} states that for at all but at most $Q^{\frac{1}{10}}$ primes $q\leq Q$, the condition $2(\mr\cdot \bu_q) \equiv 0 \pmod 1$, implies that there is a $j\leq k$ for which $|r_j|>k^5$. As we chose $N\leq \frac{T}{2(\log T)^3}$, we have that
 \begin{align*}
     N\log^2 N < \frac{T}{2(\log T)^3}\log ^2 T = \frac{T}{2\log T} \leq \pi(T) =k.
 \end{align*}
 
 It follows that $k^4N\log^2N<k^5$ and therefore, the conditions for Corollary \ref{odd cor} to hold are satisfied, and we conclude that for all of these primes $q$, we must indeed have that $\gg\frac{q-1}{2N^k}$ odd characters such that 
 \begin{equation*}
     |\chi(p)-1|\ll |\theta_p|\ll \frac{1}{N}.
 \end{equation*}
 This proves the proposition.
 \end{proof}

Finding these 1-pretentious characters plays a key role in the proof of Theorem \ref{big thm}, as such characters will provide us large character sums.

\section{Preliminary estimates}\label{chap estimates}

Before diving into the proof of Theorem \ref{big thm}, we gather in this section some estimates on exponential sums and smooth numbers that will be of use in section \ref{proof thm}.

\subsection{Some estimates on exponential sums}
What stands out when investigating logarithmic exponential sums of the form
\begin{equation}\label{alpha exp}
    \sum_{n\in I} \frac{e(\pm \alpha n)}{n}
\end{equation}
is that all the action occurs when $n$ is around $\frac{1}{\alpha}$. As we will see, this will have a direct impact on the logarithmic character sums that we evaluate in Theorem \ref{big thm}.\\

We start with a technical lemma that will allow us to handle the error terms in Lemma \ref{tail} and Lemma \ref{main constant}.\\
    
    \begin{lem}\label{fractional}
    Let $\alpha\in (0,1)$ and let $Y\geq 1$, then
    \begin{equation*}
        \int_{Y}^{\infty}\frac{\{t\}e(\pm\alpha t)}{t} dt \ll 1 + \frac{1}{\alpha Y} 
    \end{equation*}
    
    \begin{proof}
     \begin{align*}
        \int_{Y}^{\infty}\frac{\{t\}e(\pm\alpha t)}{t} dt & = \sum_{n\geq \lfloor Y \rfloor} \int_0^1
        \frac{te(\pm\alpha(t+n))}{t+n} dt + O\left(\frac{1}{Y}\right)\\
         & = \int_0^1 t e(\pm\alpha t)\left(\sum_{n\geq \lfloor Y \rfloor} \frac{e(\pm\alpha n)}{t+n}\right) dt+O\left(\frac{1}{Y}\right).
     \end{align*}
     Observe that 
     \begin{equation*}
         \bigg|\frac{e(\pm\alpha n)}{t+n} - \frac{e(\pm\alpha n)}{n}\bigg| \leq \bigg|\frac{1}{n+1} - \frac{1}{n}\bigg| \leq \frac{1}{n^2},
     \end{equation*}
     and therefore
     \begin{equation*}
         \int_{Y}^{\infty}\frac{\{t\}e(\pm\alpha t)}{t} dt = \int_0^1 t e(\pm\alpha t)dt \left(\sum_{n\geq \lfloor Y \rfloor} \frac{e(\pm\alpha n)}{n} + O\left(\frac{1}{n^2}\right)\right).
     \end{equation*}
    Now it is not hard to see that the integral on the right hand side is bounded by 1 and by partial summation, we have that
     \begin{align*}
         \sum_{n\geq \lfloor Y \rfloor} \frac{e(\pm\alpha n)}{n} & = \int_Y^{\infty} \sum_{n\leq t} e(\pm\alpha n) \frac{dt}{t^2} + O(1)\\
         & =  \int_Y^{\infty}\frac{e(\pm\alpha (\lfloor t\rfloor +1))-1}{e(\pm\alpha)-1}\frac{dt}{t^2} + O(1)\ll \frac{1}{\alpha} \int_Y^{\infty} \frac{1}{t^2}dt + 1  \ll \frac{1}{\alpha Y} + 1.
     \end{align*}
     Putting this together, it follows that 
     \begin{equation*}
          \int_{Y}^{\infty}\frac{\{t\}e(\pm\alpha t)}{t} dt \ll 1+ \frac{1}{\alpha Y}. \qedhere
     \end{equation*}
    \end{proof}
\end{lem}
The next lemma emphasizes that most contributions to (\ref{alpha exp}) happen around $\frac{1}{\alpha}$ by showing that the tail of the sum is negligeable.\\

\begin{lem}\label{tail}
Let $\alpha \in (0,1)$, then
\begin{equation*}
    \sum_{n\geq \frac{1}{\alpha}} \frac{e(\pm \alpha n)}{n}  =  \sum_{\frac{1}{\alpha} < n \leq \frac{|\log \alpha|^c}{\alpha}} \frac{e(\pm \alpha n)}{n} + O\left(\frac{1}{|\log \alpha|^c}\right).
\end{equation*}
\end{lem}
\begin{proof}
 \begin{align*}
       \sum_{n\geq \frac{1}{\alpha}} \frac{e(\pm \alpha n)}{n}& - \sum_{\frac{1}{\alpha} \leq n \leq \frac{|\log \alpha|^c}{\alpha}} \frac{e(\pm \alpha n)}{n} =\sum_{n> \frac{|\log \alpha|^c}{\alpha}} \frac{e(\pm\alpha n)}{n}\\
       &= \int_{\frac{|\log \alpha|^c}{\alpha}}^{\infty}\frac{e(\pm\alpha t)}{t} d(t-\{t\})\\
        & =  \int_{\frac{|\log \alpha|^c}{\alpha}}^{\infty}\frac{e(\pm\alpha t)}{t} dt -  \int_{\frac{|\log \alpha|^c}{\alpha}}^{\infty}\frac{e(\pm\alpha t)}{t}d\{t\}\\
        & =   \int_{\frac{|\log \alpha|^c}{\alpha}}^{\infty}\frac{e(\pm\alpha t)}{t} dt \mp 2\pi i \alpha \int_{\frac{|\log \alpha|^c}{\alpha}}^{\infty}\frac{\{t\}e(\pm\alpha t)}{t}dt + O\left(\frac{\alpha}{|\log \alpha|^c}\right).
    \end{align*}\\
     Now the second term is $O(\alpha)$ by Lemma \ref{fractional} and noticing that
     \begin{align*}
         \int_n^{n+1} \frac{e(\pm w)}{w}dw 
        &=  \int_n^{n+1} \frac{e(\pm w)}{n}dw + \int_n^{n+1} e(\pm w) \left(\frac{1}{w}-\frac{1}{n}\right)dw \\
        &= \frac{1}{n} \int_0^1 e(\pm w)dw + \int_n^{n+1} e(\pm w)\left(\frac{n-w}{nw}\right) dw  \ll \frac{1}{n^2}
 \end{align*}
    allows us to deduce that
    \begin{align*}
       \int_{\frac{|\log \alpha|^c}{\alpha}}^{\infty}\frac{e(\pm\alpha t)}{t} dt &=    \int_{|\log \alpha|^c}^{\infty} \frac{e(\pm w)}{w} dw \ll \sum_{n=|\log \alpha|^c-1}^{\infty} \frac{1}{n^2} \ll \frac{1}{|\log \alpha|^c}.\qedhere
    \end{align*} 
 \end{proof}

Analogously, it is easy to see that the beginning of the following sum does not contribute too much.\\

\begin{lem}\label{head}
\begin{equation*}
        \sum_{n\leq \frac{1}{\alpha}} \frac{1-e(\pm \alpha n)}{n} =\sum_{ \frac{1}{\alpha|\log \alpha|}< n\leq\frac{1}{\alpha}}  \frac{1-e(\pm \alpha n)}{n} + O\left(\frac{1}{|\log \alpha|}\right)
\end{equation*}
\end{lem}
 \begin{proof}
    \[
        \sum_{n\leq \frac{1}{\alpha|\log \alpha|}}  \frac{1-e(\pm \alpha n)}{n}\ll sum_{n\leq \frac{1}{\alpha|\log \alpha|}} \frac{\alpha n}{n}\ll \alpha \frac{1}{\alpha|\log \alpha|}\ll \frac{1}{|\log \alpha|}. 
        \qedhere
        \]
 \end{proof}

Interestingly, putting the sums in Lemma \ref{head} and Lemma \ref{tail} together gives rise to a constant. This will play an important role for the proof of Theorem \ref{thm even}.\\

\begin{lem}\label{main constant}
    Let $\alpha \in (0,1)$, then 
    \begin{equation*}
        \sum_{n\leq \frac{1}{\alpha}}\frac{1- e(\pm \alpha n)}{n} - \sum_{n > \frac{1}{\alpha}} \frac{e(\pm \alpha n)}{n} = \log(2\pi) + \gamma \mp \frac{i\pi}{2} +O\left(\alpha |\log \alpha|\right).
    \end{equation*}

    \begin{proof}     We have
 { \small  \begin{align}\label{4 int}
    \nonumber \sum_{n\leq \frac{1}{\alpha}}\frac{1- e(\pm \alpha n)}{n} &- \sum_{n > \frac{1}{\alpha}} \frac{e(\pm \alpha n)}{n}  
         = \int_1^{\frac{1}{\alpha}} \frac{1-e(\pm\alpha t)}{t}d\lfloor t\rfloor -\int_{\frac{1}{\alpha}}^{\infty}\frac{e(\pm\alpha t)}{t}d\lfloor t \rfloor\\
    \nonumber& = \int_1^{\frac{1}{\alpha}} \frac{1-e(\pm\alpha t)}{t}d( t- \{t\}) -\int_{\frac{1}{\alpha}}^{\infty}\frac{e(\pm\alpha t)}{t}d( t- \{t\})\\
         & = \int_{1}^{\frac{1}{\alpha}}\frac{1}{t}dt - \int_1^{\infty} \frac{e(\pm\alpha t)}{t} dt +\int_1^{\frac{1}{\alpha}} \frac{1-e(\pm\alpha t)}{t}d(\{t\}) -\int_{\frac{1}{\alpha}}^{\infty}\frac{e(\pm\alpha t)}{t}d(\{t\}).
         \end{align}}
  Now, by integrating by parts the third integral and noting that $|1-e(\alpha t)|\ll \alpha t$ for $t< \frac{1}{\alpha}$, we have that
\begin{align*}
    \int_1^{\frac{1}{\alpha}} \frac{1-e(\pm\alpha t)}{t}d(\{t\}) & = \{t\} \frac{1-e(\pm\alpha t)}{t} \Bigg|_1^{\frac{1}{\alpha}}+\int_1^{\frac{1}{\alpha}} \{t\}\left(\frac{\pm2\pi i\alpha e(\pm\alpha t)}{t} + \frac{1-e(\pm\alpha t)}{t^2}\right)dt\\
    & \ll \alpha + \alpha\int_1^{\frac{1}{\alpha}} \frac{1}{t}dt + \int_1^{\frac{1}{\alpha}} \frac{\alpha t}{t^2}dt \ll \alpha |\log \alpha|.
\end{align*}
Similarly, integrating by parts the last integral  in (\ref{4 int}), we have
\begin{equation*}
    \int_{\frac{1}{\alpha}}^{\infty}\frac{e(\pm\alpha t)}{t}d(\{t\}) = \mp2\pi i \alpha \int_{\frac{1}{\alpha}}^{\infty} \frac{\{t\}e(\pm\alpha t)}{t} dt +O\left(\alpha\right),
\end{equation*}
and by Lemma \ref{fractional} with $Y = \frac{1}{\alpha}$, the integral is $\ll 1$, and we obtain
    \begin{equation*}
    \int_{\frac{1}{\alpha}}^{\infty}\frac{e(\pm\alpha t)}{t}d(\{t\}) \ll \alpha.
\end{equation*}
Going back to (\ref{4 int}), in which we rewrite the exponential integral as sine and cosine integrals, we obtain
    \begin{align*}
         \sum_{n\leq \frac{1}{\alpha}}\frac{1- e(\pm \alpha n)}{n} &- \sum_{n > \frac{1}{\alpha}} \frac{e(\pm \alpha n)}{n} 
         = \int_{1}^{\frac{1}{\alpha}}\frac{1}{t}dt - \int_1^{\infty} \frac{e(\pm\alpha t)}{t} dt +  O\left(\alpha|\log \alpha|\right)\\
        & =  \log\left(\frac{1} {\alpha}\right) - \left(\int_{2\pi \alpha}^{\infty} \frac{\cos t}{t} dt \pm i \int_{2\pi \alpha}^{\infty}\frac{\sin t}{t}dt\right) + O\left(\alpha|\log \alpha|\right).
    \end{align*}
  The cosine integrals can be estimated using the Taylor expansions and referring to \cite{H03} p.(106), we know that
   \begin{equation*}
     -  \int_x^{\infty}\frac{\cos t}{t} dt =  \gamma + \log x + \sum_{k=1}^{\infty}\frac{(-x^2)^k}{2k(2k)!},
   \end{equation*}
   hence we deduce that
 \begin{equation*}
      - \int_{2\pi\alpha}^{\infty}\frac{\cos t}{t} dt =  \gamma + \log(2\pi\alpha)+ O(\alpha^2).
 \end{equation*}
 Now it is easily seen, using the Taylor series for sine, that 
 
 \begin{equation*}
     \int_{2\pi\alpha}^{\infty} \frac{\sin t}{t} dt = \int_0^{\infty} \frac{\sin t}{t}dt + O(\alpha),
 \end{equation*}
 and it is known (see for example \cite{abra} p.232) that
 \begin{equation*}
      \int_0^{\infty} \frac{\sin t}{t}dt = \frac{\pi}{2}.
 \end{equation*}
 Putting this together, we reach the conclusion that
\begin{align*} 
          \sum_{n\leq \frac{1}{\alpha}}\frac{1- e(\pm \alpha n)}{n} - \sum_{n > \frac{1}{\alpha}} \frac{e(\pm \alpha n)}{n} &= \log\left(\frac{1}{\alpha}\right) + \log(2\pi \alpha) + \gamma \mp \frac{i\pi}{2} + O\left(\alpha|\log \alpha|\right)\\
        &= \log(2\pi) + \gamma \mp \frac{i\pi}{2} + O\left(\alpha|\log \alpha|\right),
    \end{align*}
    as desired.
\end{proof}
\end{lem}

\subsection{Some estimates on smooth numbers}

We start this section with an estimate showing that the tail of a logarithmic sum over $y$-smooth integers is small. This will help us bound the error term in the proof of Theorem \ref{big thm}.  The argument follows the proof of Lemma 3.2 in \cite{frequency}.  

\begin{lem}\label{end 1/n}
Let $y\geq 100$, then
\begin{equation*}
           \sum_{\substack{n > y^{\log\log y}  \\ P(n)\leq y}}\frac{1}{n} \ll \frac{1}{(\log y)^{\log_3 y -3/2}}.
    \end{equation*}

\end{lem}

\begin{proof}
We have
\begin{align*}
      \sum_{\substack{n > y^{\log\log y} \\ P(n)\leq y}}\frac{ 1}{n}&\ll  \sum_{\substack{y^{\log\log y} < n\leq y^{\sqrt{\log y}} \\ P(n)\leq y}}\frac{1}{n} + \sum_{\substack{n\geq y^{\sqrt{\log y}} \\ P(n)\leq y}}\frac{1}{n} .
      \end{align*}
      For the first sum of the right hand side, we use Lemma \ref{smoothlog} and    \eqref{u^u} to get     
      \begin{align*}
          \sum_{\substack{y^{\log\log y} < n\leq y^{\sqrt{\log y}} \\ P(n)\leq y}}\frac{1}{n} &\ll \log y \int_{\log\log y}^{\sqrt{\log y}} \rho(u) du + \rho(\log\log y) \\
          &\ll(\log y)^{3/2}\rho(\log\log y) \ll \frac{\log y\sqrt{\log y}}{(\log\log y )^{\log \log y}}\ll \frac{1}{(\log y)^{\log_3 y -3/2}}.
      \end{align*}
For the second sum, given $\epsilon =  \frac{1}{\log y}$, we have      
    \begin{align*}
     \sum_{\substack{n\geq y^{\sqrt{\log y}} \\ P(n)\leq y}}\frac{1}{n}
         &\leq \sum_{\substack{n\geq y^{\sqrt{\log y}} \\ P(n)\leq y}}\frac{1}{n}\left(\frac{n}{ y^{\sqrt{\log y}}} \right)^{\epsilon}\\
        &\leq e^{-\sqrt{\log y}}\sum_{P(n)\leq y}\frac{1}{n^{1-\epsilon}}\leq e^{-\sqrt{\log y}}\prod_{p\leq y} \left(1-\frac{1}{p^{1-\epsilon}}\right)
    \end{align*}
As $p^{\epsilon} = 1 + O\left(\frac{\log p}{\log y}\right)$ for $p \leq y$, we have 

\begin{align*}
    \sum_{p\leq y} \frac{1}{p^{1-\epsilon}}- \sum_{p\leq y} \frac{1}{p} 
    &\ll \sum_{p\leq y} \frac{1}{p}\left(\frac{\log p}{\log y}\right)\\
    &\ll 1
\end{align*}
and thus, for $y$ large enough, putting this together we deduce that
\begin{equation*}
      \sum_{\substack{y^{\log \log y} < n\leq z \\ P(n)\leq y}}\frac{ \chi(n) e(n\alpha)}{n}\ll \frac{1}{(\log y)^{\log_3 y -3/2}}.
\end{equation*}
\end{proof}

Even though smooth numbers are often major allies in evaluating sums over integers, they can also be an obstacle to our ability to evaluate sums. The following lemma shows that on small intervals, the smoothness condition can be removed. 

\begin{lem}\label{remove smooth}
Let $y\geq 2$ and let $f(t)$ be a differentiable bounded function  on any interval $I \subset\left[\frac{y^B}{\log y}, y^B(\log y)^c\right]$, then for $B< \exp((\log y)^{3/5 -\epsilon})$ and $c\geq 0$ we have
 \begin{equation*}
    \sum_{\substack{n\in I\\P(n)\leq y}} \frac{f(n)}{n} = \rho(B)\sum_{n\in I} \frac{f(n)}{n} +O\left(\frac{\rho(B)\log(B+1)(\log\log y)^2 }{\log y}\right).
\end{equation*}
 \end{lem}
 
\begin{proof}
   Let $I$ be any subinterval of $\left[\frac{y^B}{\log y}, y^B(\log y)^c\right]$. By partial summation we have 
    \begin{align*}
        \sum_{\substack{n\in I\\P(n)\leq y}} \frac{f(n)}{n} &= \int_{ I} \frac{f(t)}{t} d(\psi(t,y))  = \int_I \frac{f(t)}{t} d\left(t\rho(u)\left(1 + O\left(\frac{\log u}{\log y}\right)\right)\right). 
    \end{align*}
Now for $t$ in that range we have that $\log u = O(\log B)$ and by Lemma \ref{rho}, $\rho(u) = \rho(B) + O\left(\frac{\rho(B)\log (B+1) \log\log y}{\log y}\right)$, therefore

\begin{align*}
    \sum_{\substack{n\in I\\P(n)\leq y}} \frac{f(n)}{n} &=  \left(\rho(B) + O\left(\frac{\rho(B)\log (B+1) \log\log y}{\log y}\right)\right)\int_I \frac{f(t)}{t} dt.
\end{align*}
On the other hand, using partial summation again, we have

\begin{align*}
    \sum_{n\in I} \frac{f(n)}{n}& = \frac{f(t)}{t}(t + O(1))\Big|_I - \int_I \left(\frac{f^{\prime}(t)}{t} - \frac{f(t)}{t^2}\right)(t+ O(1))dt\\
    &= f(t)\Big|_I +\int_I \frac{f(t)}{t} - f^{\prime}(t) dt + O\left(\frac{(\log\log y)^2}{\log y}\right)\\
    & =  f(t)\Big|_I -  f(t)\Big|_I + \int_I \frac{f(t)}{t} dt  + O\left(\frac{(\log\log y)^2}{\log y}\right)\\
    & = \int_I \frac{f(t)}{t} dt + O\left(\frac{(\log\log y)^2}{\log y}\right).
\end{align*}

\noindent Hence comparing both sides, we deduce that 

\begin{align*}
     \sum_{\substack{n\in I\\P(n)\leq y}} \frac{f(n)}{n}  &= \left(\rho(B) + O\left(\frac{\rho(B)\log (B+1) \log\log y}{\log y}\right)\right)\left(\sum_{n\in I} \frac{f(n)}{n} + O\left(\frac{(\log\log y)^2}{\log y}\right)\right)\\
     &=  \rho(B)\sum_{n\in I} \frac{f(n)}{n} +O\left(\frac{\rho(B)\log(B+1)(\log\log y)^2 }{\log y}\right),
\end{align*}
which ends the proof of the lemma
\end{proof}

 As we undergo the proof of Theorem \ref{thm even}, we will have to face such a sum and Lemma \ref{remove smooth} will come in handy. We are now ready for the proof of our main theorem.


\section{Proof of the main theorem}\label{proof thm}

In the following, we let $y = \log q$, $\alpha = \frac{1}{y^B}$ for some $0\leq B\leq \frac{\log\log\log}{\log\log \log \log q}$ and we let $z = q^{11/21}$. P\'olya's Fourier expansion  gives
\begin{align}\label{original}
    \sum_{n\leq \alpha q}\chi(n) &= \frac{\tau(\chi)}{2\pi i}\sum_{1\leq |n| \leq z} \overline{\chi}(n)\frac{1- e(-\alpha n)}{n} + O\left(\frac{q\log q}{z}\right),
\end{align}
    where $|\tau(\chi)|= \sqrt q$.
    For $1\leq y\leq z$ and $\delta\in  t[\tfrac{1}{\log y}, 1 ]$ we define
\begin{equation}\label{Adelta}
    A_{\delta}  = \left\{\chi \pmod q:\bigg|\sum_{\substack{1\leq |n|\leq z \\ P(n)>y}}\frac{\chi(n)}{n}\left(1 - e(-\alpha n)\right)\bigg| \leq e^{\gamma}\delta \right\}.
\end{equation}

We believe that the bound in (\ref{Adelta}) should hold for all characters modulo $q$, for $q$ large enough,  as we saw in Conjecture \ref{conjj}.  If Conjecture \ref{conjj} holds then the proof shows that Theorem \ref{big thm} is best possible for most prime moduli $q$, as the inequality sign then becomes an equality sign. For the purpose of our proof,  Theorem 4.2 in \cite{frequency} states that 
\begin{equation}\label{thm 4.2}
    \#\{\chi \pmod q :  \chi \notin A_{\delta}\}\ll q^{1-\frac{\delta^2}{\log\log q}} + q^{1-\frac{1}{500\log\log q}}.
\end{equation}

We only need $\delta=1$ for the case of odd characters. However the main term in Theorem \ref{thm even} is much smaller so we have to be a little more delicate with the choice of $\delta$, taking $\delta$ to be of size $\frac{\log\log y}{\sqrt{\log y}}$.

We now restrict our attention to characters in $A_{\delta}$ and  split the remaining sum as
\begin{equation}\label{sum split}
    \sum_{\substack{1\leq n \leq z\\P(n)\leq y}} \frac{\chi(n)}{n}\left(1- e(\pm\alpha n) \right) = S_1 + S_2^{\pm} + S_3^{\pm},
\end{equation}
where the sum
\begin{equation*}
    S_1 =  \sum_{\substack{n \geq y^B\\P(n)\leq y}}\frac{\chi(n)}{n}
    \end{equation*}
 will give the main contribution in the odd character case, the sum
\begin{equation*}
    S_2^{\pm} =  \sum_{\substack{\frac{y^B}{\log y}\leq n \leq y^B\\P(n)\leq y}}\chi(n)\frac{1-e(\pm\alpha n)}{n} -  \sum_{\substack{y^B\leq n \leq y^B(\log y)^5\\P(n)\leq y}}\chi(n) \frac{e(\pm\alpha n)}{n}
\end{equation*}
will give the main term in the even character case, and finally
\begin{align}\label{s3split}
    S_3^{\pm} &=   \sum_{\substack{1\leq n \leq \frac{y^B}{\log y}\\P(n)\leq y}}\chi(n) \frac{1- e(\pm\alpha n)}{n} -\sum_{\substack{y^B(\log y)^5\leq n \leq y^{\log\log y}\\P(n)\leq y}}\chi(n)\frac{e(\pm\alpha n)}{n} \nonumber \\ & \qquad - \sum_{\substack{y^{\log\log y}\leq n \leq z \\P(n)\leq y}}\chi(n)\frac{e(\pm\alpha n)}{n} -  \sum_{\substack{n \geq z\\P(n)\leq y}}\frac{\chi(n)}{n}
\end{align}
will contribute the error term.

\subsection{$S_3$ : Ranges with small contribution}
In this section, we dissect $S_3^{\pm}$ to show that it provides only a small contribution to \ref{sum split}.  

\begin{prop}\label{s3}
For $y$ large enough and $ 1\leq B< \exp((\log y)^{3/5 -\epsilon})$, we have
\begin{equation*}
    S_3^{\pm} \ll \frac{\sqrt{B}}{\log y}.  
\end{equation*}
Further, if $\chi \in A_{\pm}(N,T)$, then for $0<B<1$
\begin{equation*}
    S_3^{\pm} \ll \log\log y + h(T)\log y
\end{equation*}
\end{prop}

We treat the sums in $S_3^{\pm}$ one at a time,  Lemma \ref{I1} dealing with the first sum, Lemma \ref{I3}  the second   and the last two sums  in Lemma \ref{end}.
First we have

\begin{lem}\label{I1}
  Let $y\geq 2$, $B < \exp((\log y)^{3/5 -\epsilon})$ and let $\alpha = \frac{1}{y^B}$, then
  \begin{equation*}
      \sum_{\substack{1 \leq n \leq \frac{y^B}{\log y}\\P(n)\leq y}}\chi(n) \frac{1-e(\pm\alpha n )}{n} \ll \frac{\rho(B)}{\log y}.
  \end{equation*}
\end{lem}
\begin{proof} We have $|\alpha n| < 1$ since $\alpha = \tfrac{1}{y^B}$,  and thus
$\tfrac{1-e(\pm \alpha n)}{n} \ll \tfrac{\alpha n}{n}= \alpha$. 
  Therefore, as each $|\chi(n)|\leq 1$, 
  \[
      \sum_{\substack{1\leq n \leq \frac{y^B}{\log y} \\ P(n)\leq y }}\chi(n) \frac{1-e(\pm\alpha n )}{n}
      \ll \alpha\sum_{\substack{n\leq \frac{y^B}{\log y}\\P(n)\leq y}} 1   \ll \frac{\rho\left(B-\frac{\log\log y}{\log y}\right)}{\log y } \ll \frac{\rho(B)}{\log y}
  \]
  by Lemma \ref{rho}.
\end{proof}

The second sum in $S_3^{\pm}$ requires the use of a result from De la Bret\`eche for exponential sums with multiplicative coefficients over smooth numbers \cite{DLB}. We obtain 

\begin{lem}\label{I3}
Let $y\geq 2$ and let $c\geq 5$. For $\alpha = \frac{1}{y^B}$, if $B\geq 1$ then
\begin{equation*}
    \sum_{\substack{y^{B}(\log y)^c \leq n \leq y^{\log\log y}\\P(n)\leq y}} \frac{\chi(n)e(\pm\alpha n)}{n} \ll \frac{\sqrt{B}}{\log y}.
\end{equation*}
If further $\chi \in A_{\pm}(N,T)$, then for $0<B<1$
\begin{equation*}
     \sum_{\substack{y^{B}(\log y)^c \leq n \leq y^{\log\log y}\\P(n)\leq y}} \frac{\chi(n)e(\pm\alpha n)}{n} \ll \log\log y + h(T) \log y
\end{equation*}

\end{lem}

In order to prove Lemma \ref{I3}, we use the following lemma which appears as Proposition 1 in \cite{DLB}.

\begin{thm}\label{DLB}
Let $f(n)$ be a multiplicative function with $\sum_{n\leq t}|f(n)|^2\leq A^2t$, and suppose that there is $(a,m)=1$ such that $|\alpha - \frac{a}{m}| \leq \frac{1}{m^2}$ then

\begin{equation*}
    \sum_{\substack{n\leq x\\P(n)\leq y}} f(n)e(\alpha n) \ll A^2x\sqrt{\log x}\log y\left(\frac{\sqrt{y}}{\sqrt{x}} + \frac{\sqrt{m}}{\sqrt{x}}+ \frac{1}{\sqrt{m}} + e^{-\sqrt{\log x}}\right).
\end{equation*}
\end{thm}

\begin{cor}
 Let $\alpha = \frac{1}{y^B}$ for $B> 0$ and let $m$ be the closest integer to $y^B$. Le $\kappa = \max\{1,B\}$ and write $x = y^{\kappa+v}$ for $ v\geq \frac{c\log\log y}{\log y}$. Then if $v\leq B$,
 \begin{equation*}
    \sum_{\substack{n\leq x\\P(n)\leq y}} f(n)e(\pm\alpha n) \ll A^2x\sqrt{\log x}\log y \left(\frac{1}{y^{\frac{v}{2}}}+e^{-\sqrt{\log x}}\right),
\end{equation*}
 whereas if $v>B$, then 
 \begin{equation*}
    \sum_{\substack{n\leq x\\P(n)\leq y}} f(n)e(\pm\alpha n) \ll A^2x\sqrt{\log x}\log y \frac{1}{y^{\frac{B}{2}}} .
\end{equation*}
\end{cor}

A simple use of partial summation and the results just stated allow us to deduce Lemma \ref{I3}.

\begin{proof}[\textbf{Proof of Lemma \ref{I3}}]
Let $\kappa = \max\{1,B\}$. Given $\alpha = \frac{1}{y^B}$, taking $m$ to be the closest integer to $y^B$, we can apply Theorem \ref{DLB} with $A=1$. That is, we have
\begin{align*}
    \sum_{\substack{y^\kappa(\log y)^c \leq n \leq y^{\log\log y}\\P(n)\leq y}} \frac{\chi(n)e(\pm\alpha n)}{n} 
    &= \frac{1}{y^{\log\log y}}\sum_{\substack{n\leq y^{\log\log y}\\P(n)\leq y}} \chi(n)e(\pm\alpha n) - \frac{1}{y^\kappa(\log y)^c}\sum_{\substack{n\leq y^\kappa(\log y)^c\\P(n)\leq y}}\chi(n)e(\pm\alpha n) \\
    &+  \int_{y^\kappa(\log y)^c}^{y^{\log\log y}} \sum_{\substack{n\leq t\\P(n)\leq y}} \chi(n)e(\pm\alpha n) \frac{dt}{t^2}\\
    &\ll \frac{\sqrt{\log\log y}(\log y)^{3/2}}{y^{B/2}} + \frac{\sqrt{\kappa}}{(\log y)^{\frac{c}{2}-\frac{3}{2}}}\\ 
    &+ y^{\kappa/2}\log y\int_{y^\kappa(\log y)^c}^{y^{2B}} \frac{\sqrt{\log t}}{t^{3/2}} dt + \log y\int_{y^\kappa(\log y)^c}^{y^{2B}} \frac{\sqrt{\log t}}{t}e^{-\sqrt{\log t}}dt\\ 
    &+\frac{\log y}{y^{B/2}} \int_{y^{2B}}^{y^{\log\log y}}\ \frac{\sqrt{\log t}}{t} dt .
   \end{align*}
Computing the integrals gives
   \begin{align*}
     \sum_{\substack{y^\kappa(\log y)^c \leq n \leq y^{\log\log y}\\P(n)\leq y}} \frac{\chi(n)e(\pm\alpha n)}{n}  &\ll\frac{\sqrt{\kappa}}{(\log y)^{\frac{c}{2}-\frac{3}{2}}} + Be^{-\sqrt{B\log y}}\log^2 y  + \frac{\log y}{y^{B/2}}(\log y\log\log y)^{3/2}\\
    &\ll \frac{\sqrt{\kappa}}{\log y},
\end{align*}
whenever $c\geq 5$, proving the first part of the lemma when $B\geq 1$.

Now if $B<1$, then $\kappa$ is 1 and we still have to estimate the sum over the range $[y^B(\log y)^c, y(\log y)^c]$. In that case, write
  \begin{align*}
          \sum_{\substack{y^B(\log y)^c \leq n \leq y(\log y)^c\\P(n)\leq y}} \frac{\chi(n)e(\pm\alpha n)}{n} &= \sum_{\substack{y^B(\log y)^c \leq n \leq y\\P(n)\leq y}} \frac{\chi(n)e(\pm\alpha n)}{n} + \sum_{\substack{y \leq n \leq y(\log y)^c\\P(n)\leq y}} \frac{\chi(n)e(\pm\alpha n)}{n}.
    \end{align*}
As the first sum is $y$-smooth, we can remove the smoothness condition, and using Proposition \ref{switch to f} and partial summation, we obtain
\begin{align*}
    \sum_{\substack{y^B(\log y)^c \leq n \leq y\\P(n)\leq y}} \frac{\chi(n)e(\pm\alpha n)}{n} &= \sum_{y^B(\log y)^c \leq n \leq y} \frac{e(\pm\alpha n)}{n} + O(h(T)\log y)\\
    &= O\left( \frac{1}{(\log y)^c} + h(T)\log y \right).
\end{align*}
Now for the second sum, bounding trivially the numerator gives
\begin{align*}
    \sum_{\substack{y \leq n \leq y(\log y)^c\\P(n)\leq y}} \frac{\chi(n)e(\pm\alpha n)}{n}&= \sum_{\substack{y \leq n \leq y(\log y)^c\\P(n)\leq y}} \frac{1}{n} \\
    &\ll \log\log y, 
\end{align*}
therefore, putting this together, we get
\begin{equation*}
    \sum_{\substack{y^B(\log y)^c \leq n \leq y(\log y)^c\\P(n)\leq y}} \frac{\chi(n)e(\pm\alpha n)}{n} \ll \log\log y + h(T)\log y,
\end{equation*}
 which proves the second part of the Lemma.
\end{proof}

The next lemma deals with the two last sums of (\ref{s3split}) and follows directly from Lemma \ref{end 1/n}.
\begin{lem}\label{end}
Let $\chi$ be a character modulo $q$, let $\alpha$ be any real number in $(0,1]$ and $z \geq y^{\log\log y}$, then 
    \begin{equation*}
           \sum_{\substack{y^{\log\log y} < n\leq z \\ P(n)\leq y}}\frac{ \chi(n) e(\pm\alpha n)}{n} + \sum_{\substack{n \geq z\\P(n)\leq y}}\frac{\chi(n)}{n} \ll \frac{1}{(\log y)^{\log_3 y -3/2}}.
    \end{equation*}
    
\end{lem}
Finally, putting Lemmas \ref{I1}, \ref{end} and \ref{I3} together gives Proposition \ref{s3}.

\subsection{$S_1$ and $S_2$: The main contributions}

Our strategy in order to evaluate $S_1$ and $S_2$ will be to use characters that pretends to be 1, so that $\chi \in A_\pm(N,T)$.This supposes that our choice of character will satisfy
\begin{equation*}
    \underset{p\leq T}{\max}|\chi(p) - 1|\ll \frac{1}{N},
\end{equation*}
and using this hypothesis brings us back to the results we derived in section \ref{chap pretends}. As a consequence of Proposition \ref{s1}, we first evaluate $S_1$,  obtaining

\begin{prop}\label{s1=}
    Let $\chi$ be in $A_\pm(N,T)$, then for $y\geq T$ with $\log\log y = \left(1+\frac{1}{N}\right)\log\log T$ and $B<\exp((\log y)^{3/5 -\epsilon})$, we have
    \begin{equation*}
        S_1  = \log y \int_{B}^{\infty} \rho(u) du  +  O(1 + h(T)\log y).
    \end{equation*}
\end{prop}
This constitutes our main term in Theorem \ref{big thm} and it remains to evaluate $S_2$.

\subsubsection{The constant arising from $S_2$}

We show that that if $\chi$ pretends to be 1, then $S_2$ gives rise to a constant.

\begin{prop}\label{s2+-}
Let $y\geq 2$, let $0\leq B < \exp((\log y)^{3/5 -\epsilon}$ and let $\chi$ be in $A_\pm(N,T)$. Then for $\kappa = \max\{1,B\}$

    \begin{equation*}
        S_2^{\pm} = \rho(B)\left(\gamma + \log(2\pi) \mp\frac{i\pi}{2}\right) + O\left(h(T)\rho(\kappa-1)\log \log y + \frac{\rho(B)\log(B+1)(\log\log y)^2 }{\log y}\right).
    \end{equation*}

 \begin{proof}
     We start by using Proposition \ref{switch to f} with  $f(n) = 1-e(\pm\alpha n)$  for the first sum and $f(n) = e(\pm\alpha n)$ for the second sum to approximate $\chi$ by 1. We have
 
  {\small \begin{align*}
         \sum_{\substack{\frac{y^B}{\log y}\leq n \leq y^B\\P(n)\leq y}}\chi(n)\frac{1-e(\pm\alpha n)}{n} -  \sum_{\substack{y^B\leq n \leq y^B(\log y)^5\\P(n)\leq y}}\chi(n) \frac{e(\pm\alpha n)}{n} 
         &=  \sum_{\substack{\frac{y^B}{\log y}\leq n \leq y^B\\P(n)\leq y}}\frac{1-e(\pm\alpha n)}{n} -  \sum_{\substack{y^B\leq n \leq y^B(\log y)^5\\P(n)\leq y}} \frac{e(\pm\alpha n)}{n} \\
         &+O\left(h(y)\log y\int_{\kappa-1-\frac{\log\log y}{\log y}}^{\kappa-1+\frac{c\log\log y}{\log y}}\rho(u) du \right)\\
         & =\sum_{\substack{\frac{y^B}{\log y}\leq n \leq y^B\\P(n)\leq y}}\frac{1-e(\pm\alpha n)}{n} -  \sum_{\substack{y^B\leq n \leq y^B(\log y)^5\\P(n)\leq y}} \frac{e(\pm\alpha n)}{n}\\ 
         &+ O\left(h(T)\rho(\kappa-1)\log \log y\right),
         \end{align*}}
    where we bounded the integral with Lemma \ref{rho}. 
    Next, to evaluate the right hand side, we start by removing the smoothness condition with Lemma \ref{remove smooth} and then we throw back in the end ranges to the summations using Lemmas \ref{head} and \ref{tail}  with $c = 5\left(1-\frac{\log B}{\log|\log \alpha|}\right)$. This gives us       
          
   {\footnotesize \begin{align*}
        \sum_{\substack{\frac{y^B}{\log y}\leq n \leq y^B\\P(n)\leq y}}\chi(n)\frac{1-e(\pm\alpha n)}{n} -  \sum_{\substack{y^B\leq n \leq y^B(\log y)^5\\P(n)\leq y}}\chi(n) \frac{e(\pm\alpha n)}{n}     &=  \rho(B)\left(\sum_{\frac{y^B}{\log y} \leq n\leq y^B} \frac{1-e(\pm\alpha n)}{n} -  \sum_{y^B\leq n\leq y^{B}(\log y)^5} \frac{e(\pm \alpha n)}{n}\right)\\
           & +O\left(h(T)\rho(\kappa-1)\log \log y+\frac{\rho(B)\log(B+1)(\log\log y)^2 }{\log y}\right)\\
           &=  \rho(B)\left(\sum_{ n\leq y^B} \frac{1-e(\pm\alpha n)}{n} -  \sum_{ n\geq y^{B}} \frac{e(\pm\alpha n)}{n}\right)\\
           & +O\left(h(T)\rho(\kappa-1)\log \log y+\frac{\rho(B)\log(B+1)(\log\log y)^2 }{\log y}\right). 
    \end{align*}}
Finally, appealing to Lemma \ref{main constant}, we obtain
{\footnotesize  \begin{align*}
            \sum_{\substack{\frac{y^B}{\log y}\leq n \leq y^B\\P(n)\leq y}}\chi(n)\frac{1-e(\pm\alpha n)}{n} -  \sum_{\substack{y^B\leq n \leq y^B(\log y)^5\\P(n)\leq y}}\chi(n) \frac{e(\pm\alpha n)}{n}
           & = \rho(B)\left(\gamma + \log(2\pi) -\mp\frac{i\pi}{2}\right)\\&+ O\left(h(T)\rho(\kappa-1)\log \log y+\frac{\rho(B)\log(B+1)(\log\log y)^2 }{\log y}\right),
    \end{align*}}
which proves the proposition.
 \end{proof}
 \end{prop}

\subsection{Smooth 1-pretentious characters} \label{s1p}

 We already know from Propositions \ref{big set} and \ref{hyp odd} that there are many characters pretending to be 1. From now on, assume that $q$ is an admissible prime for the bounds to hold. 
 Now, recall that we have restricted our characters to be in the set $A_{\delta}$ defined as in (\ref{Adelta}), so we need to choose $N$ and $T$ to make sure that $A_{\delta}\cap  A_{\pm}(N,T)\neq \emptyset$. 

\begin{prop}\label{size Apm}
Let $y = \log q$, and let $A_\pm = A_\pm(T,N)$ for $N=\log y$ and $T = \frac{y}{4\log y}$. Then
\begin{equation*}
    |A_{\pm}|\gg q^{1-\frac{\log\log\log q}{(\log\log q)^2}}.
\end{equation*}
\end{prop}

\begin{proof}
We know, as stated in Propositions \ref{big set} and \ref{hyp odd}, that 
\begin{equation*}   
    |A_{\pm}(N,T)|\gg \frac{q}{N^{\frac{2T}{\log T}}}.
\end{equation*}
Now, as $y=\log q$, and given our choice $N = \log y$ and $T = \frac{y}{4\log y}$,  we have
\begin{align*}
    N^{\frac{2T}{\log T}}= \exp\left(\frac{ y}{2\log y}\frac{\log\log y}{\log T}\right) \leq \exp\left(\frac{ y\log\log y}{\log^2 y}\right)  = q^{\frac{\log\log y}{\log^2 y}},
\end{align*}
 from which, we deduce that 
\begin{equation*}
      |A_{\pm}| \gg q^{1-\frac{\log\log\log q}{(\log\log q)^2}}. \qedhere
\end{equation*}
\end{proof}

\begin{cor}\label{intersection}
 Let $A_{\delta}$, $A_{\pm}$ be the sets defined as above. If $\delta > \frac{(\log\log \log q)}{(\log \log q)^{\frac{1}{2}}}$, then 
 
 \begin{equation*}
 |A_{\delta}\cap A_{\pm}|\gg q^{1-\frac{\log\log\log q}{(\log\log q)^2}}. 
  \end{equation*}

 \begin{proof}
Let $\mathcal{A} =  \{\chi \pmod q : \chi \notin A_{\delta}\}$ be the exceptional set of $A_{\delta}$ and suppose that $\delta > \frac{(\log\log \log q)}{(\log \log q)^{\frac{1}{2}}}$, then by (\ref{thm 4.2}) we have that
\begin{align*}
      |\mathcal{A}|&\ll q^{1-\left(\frac{\log\log \log q}{\log\log q}\right)^2}.
\end{align*}
That is, using Proposition \ref{size Apm} we get
\begin{align*}
    |A_\delta \cap A_\pm| &= |A_\pm| - |A_\pm \cap \mathcal{A}| \geq |A_\pm|- |\mathcal{A}|\\
   &\gg q^{1-\frac{\log\log\log q}{(\log\log q)^2}} -  q^{1-\left(\frac{\log\log \log q}{\log\log q}\right)^2}  \gg q^{1-\frac{\log\log\log q}{(\log\log q)^2}},
\end{align*}
as claimed.
\end{proof}
\end{cor}
    
Now that we have found at least a character to work with, we finally have the ingredients we need and are ready to go forward with the proof of Theorem \ref{big thm}.

\subsection{Proof of Theorem \ref{big thm}}
We are now ready to prove our main theorem, along with Theorem \ref{thm even}.

 \begin{proof}[\textbf{Proof of Theorem \ref{big thm}}]
     
     Let $q$ be an admissible prime.  Starting with P\'olya's Fourier expansion, we have
      \begin{align*}
            \sum_{n\leq \alpha q}\chi(n) &= \frac{\tau(\chi)}{2\pi i}\sum_{1\leq |n| \leq z} \frac{\overline{\chi(n)}}{n}\left(1- e(-\alpha n) \right)+ O\left(\frac{q\log q}{z}\right),\\
\end{align*}
where we let $z= q^{\frac{11}{21}}$.

Now we let $\delta =\frac{\log\log y}{(\log y)^{\frac{1}{2}}}$ in (\ref{Adelta}), so that by Corollary \ref{intersection} $|A_{\delta}\cap A_{\pm}| \neq \emptyset$, and we choose a character $\chi$ in the intersection. We have
\begin{align*}
    \sum_{1\leq |n| \leq z} \frac{\chi(n)}{n}\left(1- e(-\alpha n) \right)  &=  \sum_{1\leq n \leq z} \frac{\chi(n)}{n}\left(1- e(-\alpha n) \right)  - \chi(-1)  \sum_{1\leq n \leq z} \frac{\chi(n)}{n}\left(1- e(\alpha n) \right)  \\
&= (S_1 + S_2^{-} + S_3^{-}) -\chi(-1)(S_1 + S_2^{+} + S_3^{+}).\\
\end{align*}
 At this point we need to treat the odd and even character cases separately. If $\chi$ is an even character, then we get cancellation of $S_1$ and we are left with a contribution from $S_2^{\pm}$ and an error term from $S_3^{\pm}$. Because the main term from $S_2^{\pm}$ is a constant, we need to take $B\geq 1$ for the error from $S_3^{\pm}$ to be small enough. With this restriction, using Propositions \ref{s3} and \ref{s2+-}, with $h(T)=h(y) = \frac{\log\log y}{\log y}$, we get  
\begin{align*}
     \sum_{1\leq |n| \leq z} \frac{\chi(n)}{n}\left(1- e(-\alpha n) \right) = i\pi\rho(B) + O\left(\frac{\rho(B-1)(\log\log y)^2 }{\log y}\right),
\end{align*}
and thus, going back to (\ref{original}), we obtain
\begin{align*}
     \sum_{n\leq \alpha q}\overline{\chi}(n) &= \frac{\tau(\chi)}{2\pi i}\left(i\pi\rho(B) + O\left(\frac{\rho(B-1)(\log\log y)^2 }{\log y}\right)\right)+ O(q^{10/21}\log q)\\
     & = \frac{\tau(\chi)\rho(B)}{2} +O\left(\frac{\sqrt{q}\rho(B-1)(\log\log y)^2 }{\log y}\right).
\end{align*}
Recalling that $y = \log q$ and that $|\tau(\chi)|= \sqrt{q}$, we get
\begin{equation*}
    \max_{\substack{\chi \neq \chi_0\\\chi \text{ even}}} \bigg|\sum_{n\leq \frac{q}{(\log q)^B}}\chi(n)\bigg|\geq  \frac{\rho(B)}{2}\sqrt{q} +O\left(\frac{\sqrt{q}\rho(B-1)(\log\log \log q)^2 }{\log \log q}\right),
\end{equation*}
as desired.
 \begin{rmk}
 Note that the restriction on $q$ is unnecessary for the even character case and that Theorem \ref{thm even} holds for any prime $q$.
 \end{rmk}
As for the odd character case, given $\chi \in A_{\delta}\cap A_{-}$, we allow $B>0$ and we use Propositions \ref{s3} for $S_3^+$ and $S_3^-$, Proposition \ref{s2+-} for $S_2^+$ and $S_2^-$ and Proposition \ref{s1=} for $S_1$, to obtain
\begin{align*}
     \sum_{1\leq |n| \leq z} \frac{\chi(n)}{n}\left(1- e(-\alpha n) \right) &= 2\log y \int_B^{\infty} \rho(u) du + 2\rho(B)(\gamma \log(2\pi))+O(\log\log y)\\
     &=2\log y \int_B^{\infty} \rho(u) du +O(\log\log y),
\end{align*}
where the error term is arising from Propositions \ref{s1} and \ref{s3}.
As a consequence, using (\ref{original}), we deduce that
\begin{align*}
    \sum_{n\leq \alpha q}\overline{\chi}(n) &= \frac{\tau(\chi)}{2\pi i}\left(2\log y \int_B^{\infty} \rho(u) du +O(\log\log y)\right)\\
    &=\frac{\tau(\chi)}{\pi i}\log y \int_B^{\infty} \rho(u) du +O(\sqrt{q}\log\log y),
\end{align*}
from which we conclude that
\begin{equation*}
    \max_{\substack{\chi \neq \chi_0\\\chi \text{ odd}}} \bigg|\sum_{n\leq \frac{q}{(\log q)^B}}\chi(n)\bigg|\geq \frac{\sqrt{q}}{\pi }\log \log q \int_B^{\infty} \rho(u) du +O(\sqrt{q}\log\log\log q),
\end{equation*}
thus proving the theorem.
 \end{proof}

\bibliographystyle{plain}
\bibliography{bibllcs.bib}

\end{document}